\documentclass{amsart}
\usepackage{amssymb}
\usepackage{amsthm}
\usepackage{thmtools}
\usepackage{macros_DS}
\standardsettings
\colorcommentstrue

\draftfalse

\renewcommand{\ss}{\mathbf s}

\newcommand{\matrixA}{A}

\long\def\comment#1{}

\renewcommand{\O}{\stopcompiling}

\renewcommand{\dist}{\operatorname{dist}}

\renewcommand{\HD}{\dim}

\theoremstyle{theorem}
\maketheorem{affinecorollary}{Affine Corollary}
\theoremstyle{definition}
\maketheorem{affinecorollaryD}{Affine Corollary}

\begin{document}

\title{Diophantine Approximation on Abelian Varieties; a Conjecture of M. Waldschmidt}
\authorlior
\authorlambert\authorkeith\authordavid

\maketitle

\begin{abstract}
Following the work of Waldschmidt, we investigate problems in Diophantine approximation on abelian varieties. First we show that a conjecture of Waldschmidt for a given simple abelian variety is equivalent to a well-known Diophantine condition holding for a certain matrix related to that variety. We then posit a related but weaker conjecture, and establish the upper bound direction of that conjecture in full generality. For rank 1 elliptic curves defined over a number field $K \subsetneqq \R$, we then obtain a weak-type Dirichlet theorem in this setting, establish the optimality of this statement, and prove our conjecture in this case. 
\end{abstract}

\section{Introduction}

\subsection{Background on Diophantine Approximation}\label{Intro_Dio}

In its most general form, the field of \textit{Diophantine approximation} usually seeks to quantify in a precise manner the density of some subset $\QQ$ within its closure $M$, with respect to some notion of complexity on the points of $\QQ$. Specifically, adapting language from \cite{FKMS1,FKMS2},
we consider a Diophantine tuple to consist of: (1) a complete metric space $(M,\dist)$; (2) a countable subset $\QQ \subset M$; and (3) a \textit{height function} $H: \QQ \to (0,\infty)$. When addressing questions in the field, we must specify the tuple $( M, \dist, \QQ, H)$. Given a function $\psi: (0,\infty) \to (0,\infty)$, we say that $x \in M$ is $\psi$-\textit{approximable} if there exists a constant $C_x$ and an infinite sequence $\QQ \ni r_ n \to x$ satisfying
\[
\dist(r_n, x) < C_x \psi(H(r_n)) \quad \text{ for all } n.
\]
Traditional questions in the field concern characterizing the functions $\psi$ for which (almost) all points are $\psi$-approximable, or for a fixed function $\psi$, to describe in some manner the set of points which are $\psi$-approximable.

Consider this definition concretely in the case of $(\R,|\cdot |, \Q, H )$ where $H(\frac{p}{q}) = q$. We call a function $\psi: (0,\infty) \to (0,\infty)$ a \emph{Dirichlet function} if for every $x$, there exists a constant $C_x$ such that $\left|x-\frac{p}{q}\right| < C_x \psi\left(H\left(\frac{p}{q}\right)\right)$. In this language, the corollary to Dirichlet's seminal theorem below can be restated as saying that the function $\psi(t) = \frac1{t^2}$ is a Dirichlet function for irrationals in $\R$ (and that the constant $C_x$ can be taken to be 1 for every $x$). 
\ignore{
Moreover, the function $1/t^2$ is \emph{optimal} in the sense that any function which decays faster, i.e., any $\varphi$ such that
\[
\lim_{t \to \infty} \frac{\varphi(t)}{1/t^2} = 0
\]
is not a Dirichlet function. \comdavid{Do we care about optimality in this paper? If not, why mention it?}
}

More generally, one can consider the problem of ``simultaneous approximation,'' which seeks to approximate a vector $\xx \in \R^m$ by rationals expressed in the form $\frac{\pp}{q}$, where $\pp \in \Z^m$, $q \in \N$, and $\gcd(\pp, q) = 1$. The height function in this case is generally taken to be $H\left(\pp/q\right) = q$. It is common, in this setting, to ``clear denominators'' and consider the expression $\left\| q\xx - \pp\right\|$ instead. This has led to consideration of the more general notion of \emph{Diophantine approximation on matrices}, where an $m\times n$ matrix $A$ is evaluated in terms of the quality of ``integer approximations'' $(\pp,\qq)\in \Z^m\times \Z^n$, where the quality is measured in terms of the tradeoff between the quality of approximation $\|A\qq+\pp\|$ and the height $\|\qq\|$. The first result along these lines is Dirichlet's theorem in Diophantine approximation:

\begin{theorem}[Dirichlet 1842]\label{dirthm}
Let $A$ be an $m\times n$ matrix and fix $Q \geq 1$. Then there exist $\pp\in\Z^m$ and $\qq\in\Z^n\butnot\{\0\}$ such that
\[
\|A\qq + \pp\| < Q^{-n/m}, \;\;\;\; \|\qq\| \leq Q.
\]
\end{theorem}

\begin{corollary}\label{dircor}
Let $A$ be an $m\times n$ matrix. Then there exist infinitely many pairs $(\pp,\qq) \in \Z^{m+n}$ such that
\[
\|A\qq + \pp\| < \|\qq\|^{-n/m}.
\]
\end{corollary}

Beyond Dirichlet's theorem, we can ask whether the quality of approximation is improved by some factor of the form $\|\qq\|^\epsilon$ with $\epsilon > 0$:

\begin{definition}\label{BA_VWA}
Let $A$ be an $m\times n$ matrix. The \emph{Diophantine exponent} of $A$ is
\[
\omega(A) := \limsup_{\|\qq\|\to\infty} \frac{-\log\min_\pp \|A\qq + \pp\|}{\log\|\qq\|}
\]
Equivalently, $\omega(A)$ is the supremum of $\omega$ such that there are infinitely many $(\pp,\qq)\in\Z^{m+n}$ such that
\[
\|A\qq + \pp\| \leq \|\qq\|^{-\omega}.
\]
\end{definition}

The corollary to Dirichlet's theorem implies that $\omega(A) \geq n/m$ for all $m\times n$ matrices $A$. Such a matrix is called \emph{very well approximable} if $\omega(A) > n/m$. A matrix is called \emph{badly approximable} if there exists a constant $C$ such that 
\[
\|A\qq + \pp\| \geq C\|\qq\|^{-n/m}
\] for every $(\pp, \qq) \in \Z^{m+n}$ with $\qq \ne \textbf{0}$.

It is worth mentioning that these results about matrices don't fit nicely in the language of a Diophantine tuple introduced above. We have decided to frame the introduction this way, however, because it is the most natural way to understand the goal of this paper, which is to describe the approximation of arbitrary points on an abelian variety by rational points, with respect to a Riemannian metric and a natural height function, discussed next. 

\subsection{Background on Abelian Varieties}

As mentioned in the abstract, the goal of this paper is to study Diophantine approximation in the context of abelian varieties, and in line with the previous section, we need to specify the tuple in which we're working. Note that our description of abelian varieties is in line with our needs to describe the set as a metric space (in particular a Riemannian manifold), and hence differs in language from other popular approaches, but the description is equivalent when focusing on the real or complex locus of the variety. 

An abelian variety is a projective variety $\AA$ together with a group law $+:\AA\times\AA\to\AA$ which is a homomorphism. The operator $+$ is always abelian \cite[p.2, p.44]{Mumford} and the pair $(\AA(\C),+)$ is isomorphic as a complex Lie group to $(\C^g/\Lambda,+)$ for some $g\in\N$ and some lattice $\Lambda \subset \C^g$ \cite[p.2, p.44]{Mumford}.


Now let $K$ denote a number field with a fixed real embedding, let $\AA$ be an abelian variety defined over $K$, and let $\mathcal{A}(K)$ denote the set of points in $\mathcal{A}$ whose coordinates are in $K$. We have an inclusion 
\[
\mathcal{A}(K) \subset \mathcal{A}(\R) \subset \mathcal{A}(\C).
\]

The real locus $\mathcal{A}(\R)$ is virtually a torus of the same dimension $g$ (when viewed as a real manifold), and carries a Riemannian metric, which we denote by $\dist$. We now have three of the four pieces of our set-up: we will study the density of the set $\mathcal{A}(K)$ in its closure inside $\mathcal{A}(\R)$, with regards to the Riemannian metric. What remains is to describe the height in this situation. 

The construction of a height function in this setting goes back to N\'eron and Tate, independently, and a succinct introduction can be found in \cite[III. Section 1]{Lang_diophantine_survey}. While the full construction is well beyond the scope of this paper, the upshot is that if we choose a symmetric, invertible, ample line bundle over $\mathcal{A}$ (see \cite[Cor. 7.2]{Milne} for a proof that such a choice of line bundle is always possible), then the corresponding height function $\hat{h}$ has zeros exactly consisting of the $K$-rational torsion points, and in fact defines a positive definite quadratic form on the finite-dimensional vector space $\R \otimes \mathcal{A}(K)$ \cite[III, Section 1, Theorem 1.3]{Lang_diophantine_survey}. We view the $K$-rational points modulo torsion as a lattice in $\R \otimes \mathcal{A}(K)$. 

For a more concrete discussion, fix an embedding into some projective space $\mathbb{P}^N$, and define a preliminary height (note this is not the final form of the height function we will use in what follows) of any point $P = [x_0, x_1, ..., x_N], x_i \in K$, to be
\[
H(P) = \left(\prod_{\nu \in M_K} \max\{|x_0|_\nu, ..., |x_N|_\nu\}^{[K_\nu:\Q_\nu]}\right)^{1/[K:\Q]}.
\]
Here $M_K$ denotes the set of standard absolute values on $K$, namely those absolute values on $K$ whose restriction to $\Q$ yields one of the usual absolute values (either the archimedean absolute value $|\cdot |_\infty$ or one of the $p$-adic absolute values), and $[K:\Q]$ denotes the degree of the field extension. 

The \textit{canonical height} of $P \in \mathcal{A}(K)$ is defined to be 
\[
\hat{h}(P) = \displaystyle\lim_{n \to \infty} \frac1{n^2} \log(H(nP)).
\]
See \cite[Chapter VIII, Section 9]{Silverman} for details. 

As before, the goal is to \textit{quantify} the density of $\mathcal{A}(K)$ inside the closure, and it seems reasonable that ``how dense'' this set is should depend in some sense on how many $K$-rational points there are. 

The celebrated Mordell--Weil theorem states that the set $\mathcal{A}(K)$ is a finitely-generated abelian group, and therefore, by the Fundamental Theorem of Finitely-Generated Abelian Groups, it can be written as 
\[
\Z^r \oplus \Z_{\text{tors}},
\]
where $r$ is called the \textit{rank} of the group, and $\Z_{\text{tors}}$ is the finite subgroup consisting of the torsion points. It is natural to conjecture that the optimal Dirichlet function for this setting should depend in some way on the rank of the abelian variety. 

In \cite{Waldschmidt}, M. Waldschmidt makes the following conjecture. Let $\AA(\R)^0$ denote the identity component of $\AA(\R)$. For a real number $h$, let 
\[
\hat{\eta}_{\mathcal{A}}(h)  = \inf \left\{ \delta : \text{ for any } \zeta \in \mathcal{A}(\R)^0, \text{ there exists } \gamma \in \mathcal{A}(K) \text{ with } \hat{h}(\gamma) \leq h \text{ and } \dist(\zeta, \gamma) < \delta \right\}.
\]
\begin{conjecture}[{\cite[Conjecture 1.1]{Waldschmidt}}]\label{Wald_conj}
    Let $\mathcal{A}$ be a simple abelian variety of dimension $g$ over a number field $K$ embedded in $\R$. Denote by $r$ the rank of the Mordell--Weil group $\mathcal{A}(K)$. For any $\varepsilon > 0 $, there exists $h_0 > 0$ (which depends only on the abelian variety $\mathcal{A}$, the real number field $K$, and $\varepsilon$) such that, for any $h \geq h_0$, 
    \[
    \hat{\eta}_\mathcal{A}(h) \leq h^{-(r/2g)+ \varepsilon}.
    \]
 \end{conjecture}

 In line with the distinction between Theorem \ref{dirthm} and Corollary \ref{dircor}, we posit a related conjecture.
Given $\kappa > 0$, we say that a point $P \in \mathcal{A}(\R)$ is $\kappa$-approximable if there exists a constant $C$ and infinitely-many $Q \in \AA(K)$ such that 
\[
\dist(P,Q) \leq C\left(\hat{h}(Q)\right)^{-\kappa}.
\]
Analogously to the case of matrices discussed above, we define the Diophantine exponent of $P \in \mathcal{A}(\R)^0$ to be 
\[
\sigma(P) := \limsup_{\hat{h}(Q) \to \infty} \frac{-\log \dist(P,Q)}{\log \hat{h}(Q)} \geq \kappa,
\]
and we define the \textit{Diophantine exponent} of the abelian variety to be 
\[
\sigma_{\mathcal{A},K} := \inf\{\sigma(P): P \in \mathcal{A}(\R)^0 \setminus \mathcal{A}(K)\}.
\]

We can now state the main result of this paper, and a related conjecture.

\begin{theorem}\label{theorem_main}
    Waldschmidt's conjecture \ref{Wald_conj} is equivalent to the statement that a certain matrix (determined by the abelain variety $\mathcal{A}$ and the number field $K$, and defined below in Section \ref{S:proof}), is not very well approximable in the sense of Definition \ref{BA_VWA}.
\end{theorem}
 \begin{conjecture}\label{Conj_FLMS}
     Let $\mathcal{A}$ be a simple abelian variety of dimension $g$, $K$ a number field with a fixed real embedding, and let $r$ denote the rank of the Mordell--Weil group $\mathcal{A}(K)$. Then 
     \[
     \sigma_{\mathcal{A},K} = \frac{r}{2g}.
     \]
 \end{conjecture}

A few comments about these conjectures are in order. Firstly, the reason for the restriction to the set $\mathcal{A}(\R)^0$ is a necessary one, as it may be that even if $r \geq 1$, the closure of $\mathcal{A}(K)$ may not coincide with $\mathcal{A}(\R)$. For example, consider the elliptic curve $y^2 = x^3 -12x - 1$, with \href{https://www.lmfdb.org/EllipticCurve/Q/110160/cd/1}{LMFDB label 110160.cd1}, considered as a curve defined over $\Q$.  This curve is rank 1 and torsion-free (i.e. $\mathcal{A}(\Q)\cong \Z$), with a possible generator of the Mordell--Weil group given by (5,8).  Because the discriminant of this curve is positive, the real locus of the curve has two connected components, and $\mathcal{A}(\R) \setminus \mathcal{A}(\R)^0$ contains no rational points. Therefore, in this case, 
\[
\overline{\mathcal{A}(\Q)} = \mathcal{A}(\R)^0 \subsetneqq \mathcal{A}(\R)
\]
The first equality in this case follows directly from our Theorem \ref{weakDT} below. 

Secondly, it's not known whether the closure of $\mathcal{A}(K)$ always includes $\mathcal{A}(\R)^0$. This result is known when the inequality $r \geq g^2-g+1$ is satisfied (\cite{Waldschmidt}), which includes all eliptic curves of rank at least 1 (since $g=1$ in this case). Both conjectures \ref{Wald_conj} and \ref{Conj_FLMS} would imply that $\mathcal{A}(\R)^0 \subseteq \overline{\mathcal{A}(K)}$. One could alternatively define $\sigma_{\mathcal{A},K}$ to only consider those points $P \in \overline{\mathcal{A}(K)}$, but we have opted for our formulation because of the relationship with conjecture \ref{Wald_conj}, and to establish a nontrivial upper bound in Lemma \ref{UB} below. In the event that $\mathcal{A}(\R)^0  \not\subset \overline{\mathcal{A}(K)}$, we simply have $\sigma_{\mathcal{A},K} = 0$ by definition. 

\ignore{
\comkeith{We need more historical development}
The study of elliptic curves is a rich field with numerous and surprising connections to many diverse areas of mathematics. One of the oldest aspects of this theory is the study of the rational points on the curve, that is, points $(x,y)$ lying on an elliptic curve $$y^2 = ax^3 + bx^2 + cx + d, \quad a,b,c,d \in \Q,$$ both of whose coordinates are rational. 

Let $E$ be an elliptic curve defined over $\Q$ and $E(\Q)$ denote its rational points. It is a classical fact that the set of real points $E(\R)$ admits a commutative group law,  which restricts to a group law on $E(\Q)$. One of the most significant theorems of the twentieth century, the Mordell--Weil Theorem, states that $$E(\Q) \text{ is a finitely generated abelian group}\footnote{In fact this theorem holds for any number field $K$}.$$ From the Fundamental Theorem of Finitely Generated Abelian Groups, we therefore know that $$E(\Q) \cong \Z^r \oplus \Z_{\text{tors}}.$$  Where the exponent $r$ is called the \emph{rank} of the elliptic curve. 

The rational points on an elliptic curve are dense if and only if its rank $r$ is greater than or equal to 1. Our aim in this paper is to {\it quantify} that density, in a sense analogous to the following result of Dirichlet:
$$\text{For every $x \notin \Q$, there exist infinitely many $\frac{p}{q} \in \Q$ such that $\left|x-\frac{p}{q}\right| < \frac1{q^2}$}.$$ 
}

The notion of equivalence in the setting of abelian varieties defined over $K$ is that of an \textit{isogeny}, which is a surjective map with finite kernel, defined over the number field $K$. Morever, every isogeny $\mathcal{A}_1 \to \mathcal{A}_2$ comes paired with a dual isogeny in the reverse direction, $\mathcal{A}_2 \to \mathcal{A}_1$ (c.f. \cite[Ch. 2, Section 8, p. 80]{Mumford}). While the isogeny and its dual have some deeper relations, for our purposes it is enough to know simply that the relation of having an isogeny between varieties is symmetric. 

\begin{lemma}
The Diophantine exponent is an isogeny invariant.
\end{lemma}
\begin{proof}
Fix an ample divisor $D$ on $\mathcal{A}_2$, and consider its associated N\'eron--Tate height $\hat{h}_{2,D}$.  Let $\phi\in\text{Mor}_K(\mathcal{A}_1(\R),\mathcal{A}_2(\R))$.  Define $\hat{h}_{1,\phi^*(D)}$ to be the N\'eron--Tate height associated with the pullback $\phi^*(D)$ of $D$ to $\mathcal{A}_1$. Note that the pullback of an ample line bundle by an isogeny is ample (\cite[Theorem 1.2.13]{Lazarsfeld}).  Let $Q\in \mathcal{A}_1(K)$ with $\hat{h}_{1,\phi^*(D)}(Q)>0$ and such that $d_{\mathcal{A}_1}(P,Q) < C \left(\hat{h}_{1,\phi^*(D)}(Q) \right)^{-\kappa}$.  There is a constant $c_1>0$, dependent only on $\text{deg}(\phi)$ and $D$, such that $\hat{h}_{2,D}(\phi(Q))\leq\text{deg}(\phi)\hat{h}_{1,\phi^*(D)}(Q)+c_1$ \cite[Chapter~VIII, Theorem~5.6]{Silverman}.  As $\phi$ lifts to an invertible linear map $d\phi$ between the tangent spaces $T_0\mathcal{A}_1$ and $T_0\mathcal{A}_2$ \cite[Chapter~8, Section~2, Subsection~27]{Bombieri_Heights}, it only distorts distances by a constant, $|d\phi|$.
\begin{align*}
\text{d}_{\mathcal{A}_2}(\phi(P),\phi(Q))&\leq|d\phi|\text{d}_{\mathcal{A}_1}(P,Q)\\
&<|d\phi|C(\hat{h}_{1,\phi^*(D)}(Q))^{-\kappa}\\
&\leq|d\phi|C\deg(\phi)^{\kappa}(\hat{h}_{2,D}(\phi(Q))-c_1)^{-\kappa}.\\
&\leq\tilde{C}(\hat{h}_{2,D}(\phi(Q)))^{-\kappa}
\end{align*}
where $\tilde{C}>0$ exists and is finite due to the fact that the kernel of $\phi$ is finite and there is an infinite set of $Q$ with $\hat{h}_1(Q)>0$ implied by $\kappa$-approximablility on $\mathcal{A}_1(\mathbb{R})$ and because there is only a finite number of rational points with height below any given threshold \cite[Theorem 5.11]{Silverman}. This shows that 
\[
\sigma_{\mathcal{A}_2,K} \leq \sigma_{\mathcal{A}_1,K},
\]
and applying the argument again with the dual isogeny and the roles of $\mathcal{A}_1$ and $\mathcal{A}_2$ reversed establishes the opposite inequality.
\end{proof}
\ignore{
\begin{corollary}
The Diophantine exponent $\sigma$ is invariant under all non-constant morphisms defined over $K$.
\end{corollary}
\comdavidl{This is accurate, but equivalent to the lemma above for simple Abelian varieties.}
}

\section{Main Theorem}

\subsection{A problem in inhomohegenous approximation}

\comment{
\comkeith{Cut this discussion, conjecture discussed above. BA defined in intro now}
One could try to prove Waldschmidt's original conjecture for rank 1 elliptic curves over $\Q$. Let $Q \in E(\Q)$ be the generator of the free part of $E(\Q)$. Waldschmidt's conjecture is that for every $P \in E(\R) \setminus E(\Q)$ and for every $\varepsilon > 0$, there exists $C$ such that for every $N \geq 1$, there exists $n \leq N$ satisfying 
\begin{equation}\label{higher_conj}
d_E(P,[n]Q) < C(N^2\hat{h}(Q))^{-1/2 + \varepsilon}.
\end{equation}
Our Theorem \ref{weakDT} above establishes this result without the $\varepsilon$ term, but with $n$, rather than $N$, in the denominator.

There has been recent work on understanding those tuples for which Equation \eqref{higher_conj} holds or fails to hold (see \cite{KleinbockWadleigh}, \cite{KimKim20}). Although the statements of the theorems are not useful for us here (because they're measure theoretic and our set of interest is a null set), studying the proofs might shed some light on obstructions to this inequality. If we can identify \emph{necessary} conditions for when a tuple does not satisfy Equation \eqref{higher_conj}, then we identify a property which our tuples must avoid.

\comkeith{Add the isogeny invariance, and the equivalence of BA or not VWA with various strengthenings of the inequality}.
}

We can give some insight on Waldschmidt's conjecture using inhomogeneous Diophantine approximation in $\R^d$. Recall the definitions of badly approximable and very well approximable matrices given in Definition \ref{BA_VWA} above. Eventually, we will show that Waldschmidt's conjecture for a pair $(\mathcal{A}, K)$ is equivalent to a certain matrix being not very well approximable. We first mention an equivalent formulation of a stronger condition on the matrix, but we omit the proof because it's essentially the same as that of Theorem \ref{VWA_equiv} below.  

Recently, a paper by paper by Moshchevitin and Neckrasov \cite{MoshchevitinNeckrasov} established results about inhomogeneous approximations which are more general than what we need here, but for the ease of the reader we provide independent proofs of these statements. The methods we use are different than those employed in \cite{MoshchevitinNeckrasov}, and the language they employ is much more general. We give the way to recover our Theorems \ref{BA_equiv} and \ref{VWA_equiv} from their work below. 

\comment{
We will need the following concepts, similar to many introduced already in the paper. 
\begin{definition}\label{BA_VWA}
We say that an $m\times n$ matrix $\matrixA$ is \textit{badly approximable}, written $\matrixA \in \BA$, if there exists some constant $c >0$ for which for every $\mathbf{p}\in \Z^m$ and $\mathbf{q}\in\Z^n\setminus\{\boldsymbol{0}\}$ the following inequality holds: 
\[
\left| \matrixA\qq - \pp\right| > \frac c{\|\qq\|^\frac{n}{m}}.
\]
We say that a matrix $\matrixA$ is $\sigma$-\textit{very well approximable} for $\sigma > 1$ if for infinitely many $\mathbf{p}\in \Z^m$ and $\mathbf{q}\in\Z^n\setminus\{\boldsymbol{0}\}$ we have
\[
\left|\matrixA\qq - \pp \right| < \frac 1 {\|\qq\|^{\frac{n}{m}\sigma}}.
\]
We say that $\matrixA$ is simply \textit{very well approximable}, written $\matrixA \in \VWA$, if it is $\sigma$-very well approximable for some $\sigma > 1$.
\end{definition}

We do have the following in this direction.
}

\begin{theorem}\label{BA_equiv}
Fix $m,n\in\N$ and let $\matrixA \in \MM_{m\times n}$. Then  $\matrixA$ is badly approximable  if and only if there exists a constant $C$ such that for all sufficiently large $Q$ and all $\boldsymbol\gamma$, there exist $\pp\in\Z^m$ and $\qq\in\Z^n$ with $0 < \|\qq\| \leq Q$ satisfying
\[
Q^\frac{n}{m}\|\matrixA\qq + \pp + \boldsymbol\gamma\| \leq C.
\]
\end{theorem}
\ignore{
\begin{proof}
To prove the forward direction, assume that $\matrixA\in\MM_{m\times n}$ is badly approximable, and let $c$ be as in  Definition \ref{BA_VWA}. Consider the unimodular lattice 
\[
\Lambda_\matrixA := \begin{pmatrix} 1&\matrixA \\ 0&1\end{pmatrix} \Z^{m+n}
\]
and the flow on the space of lattices given by 
\[
g_t := \begin{pmatrix} e^{t/m} I_m &0\\0&e^{-t/n} I_n\end{pmatrix}.
\]
By the definition of badly approximable, it follows that none of the nonzero lattice vectors in $g_t \Lambda_\matrixA$ have supremum norm less than $c^\frac{1}{m+n}$.

By Mahler's Compactness Theorem, since the lattices $\{g_t \Lambda_\matrixA\}$ have nonzero shortest length bounded from below, this set is precompact in the topology on the space of unimodular lattices. The function assigning to a lattice $\Lambda$ its \textit{codiameter} (the diameter of the corresponding homogeneous space $\R^{m+n}/\Lambda$) is continuous on the space of lattices, so its maximum is bounded on any precompact set. Therefore, the codiameters of the lattices $\{g_t \Lambda_\matrixA\}$ are uniformly bounded. 

Now consider the affine grids obtained by translating a lattice $\Lambda$ by some vector $\yy$. In the quotient space $\R^{m+n}/\Lambda$, every element of the grid $\Lambda + \yy$ is mapped to the projection of the vector $\yy$. Distances in the projection are given by 
\[
d_{\R^{m+n}/\Lambda}(\xx,\yy) = \inf_{\rr \in \Lambda}\|\xx - \yy + \rr\|.
\]
This means that the length of the shortest nonzero vector in the affine grid is equal to the distance from $\0$ to $\yy$ in the quotient space $\R^{m+n}/\Lambda$, which is bounded from above by the codiameter of $\Lambda$. 

Combining this observation with the Mahler's argument above implies that for any $\boldsymbol\gamma\in\R^m$ the affine grid 
\[
g_t\Lambda_\matrixA + \begin{pmatrix} \boldsymbol\gamma\\0 \end{pmatrix}
\] 
has shortest nonzero vector not greater in length than some uniform constant $C_2$. 

Fix $Q$ as in the statement of the theorem, and let $t = n\log(Q/C_2)$. Consider the grid $g_t \Lambda_\matrixA + \begin{pmatrix} e^{t/m} \boldsymbol\gamma \\ 0 \end{pmatrix}$. We know that there exists at least one nonzero element whose norm does not exceed $C_2$. All such elements of this grid can be written as $g_t \Lambda_\matrixA \begin{pmatrix} \pp \\ \qq \end{pmatrix} + \begin{pmatrix} e^{t/m} \boldsymbol\gamma \\ 0 \end{pmatrix}$, so we pick such a vector and compute its sup norm: 
\[
\begin{pmatrix} e^{t/m} I_m &0\\0& e^{-t/n} I_n\end{pmatrix} \begin{pmatrix} 1&\matrixA\\ 0&1\end{pmatrix} \begin{pmatrix} \pp\\ \qq\end{pmatrix} + \begin{pmatrix} e^{t/m}\boldsymbol\gamma\\ 0 \end{pmatrix} = \begin{pmatrix} e^\frac tm\left(\matrixA\qq + \pp + \boldsymbol\gamma\right)\\ e^{-t/n} \qq \end{pmatrix},
\]
where both components of this vector are less than $C_2$ in absolute value. 

Looking at the second inequality first, namely that $e^{-t/n}\|\qq\| < C_2$, we see that this means 
\[
e^{-t/n}\|\qq\| = \frac {C_2}{Q} \|\qq\| \leq C_2 \Rightarrow \|\qq\| \leq Q. 
\]
By examining the first component, we similarly see that 
\[
C_2 \geq \|e^{t/m}\left(\matrixA\qq + \pp + \boldsymbol\gamma\right)\| = \left(\frac{Q}{C_2}\right)^\frac{n}{m} \|\matrixA\qq + \pp + \boldsymbol\gamma\|
\]
which implies that $Q^\frac{n}{m}\|\matrixA\qq + \pp + \boldsymbol\gamma\| \leq C^{1+\frac{n}{m}}_2$.

To prove the complementary direction, we will actually prove the following: if $\matrixA \notin \BA$, then the set 
\[
D:= \left\{\gamma:\ \forall C,\ \exists^\infty Q \text{ such that } \forall \pp\in\Z^m,\qq\in\Z^n \, 0 < \|\qq\| \leq Q,\ \ Q^\frac{n}{m}\|\matrixA\qq + \pp + \boldsymbol\gamma\|> C \right\}
\]
is hyperplane absolutely winning for any $\beta<\frac13$. \comdavid{need to define hyperplane absolute winning} Since any hyperplane absolutely winning set is in fact uncountable and has full Hausdorff dimension, this is a priori stronger than simply proving that the set is nonempty. 

Since $\matrixA\notin\BA$, for each $k$, there exist lattice points $(u_\matrixA^T)^{-1} (p_k, q_k)^T \in \Lambda_\matrixA^*$ for which 
\[
\displaystyle\min_t \left\| ((g_t u_\matrixA)^T)^{-1} (p_k,q_k)^T\right\| < \frac1{k}.
\]
Let $t_k$ denote the unique time at which $\rr_k = (\aa_k,\bb_k) := ((g_{t_k} u_\matrixA)^T)^{-1} (p_k,q_k)^T$ satisfies $\|\aa_k\| = \|\bb_k\|$.
Without loss of generality, by extracting a subsequence if necessary, we can assume that 
\[
t_{k+1}-t_k > -2m\log(\beta)
\]
for all $k$. 

We now describe Alice's strategy, as a response to Bob's choice of ball. She deletes a hyperplane based on the following criterion:
\begin{center}
if for some $k$, the radius of Bob's ball lies between $4C\beta^{-1} e^{-\frac{t_k}{m}}$ and   $4C\beta^{-2}e^{-\frac{t_k}{m}}$,
\end{center}
then Alice will delete a specific hyperplane-neighborhood (to be determined shortly), and otherwise she can choose arbitrarily. So it remains only to explain how Alice will make her choice when that criterion is satisfied, and why such an outcome will belong to $D$.

Since $(g_{t_k} \Lambda_\matrixA)^* = ((g_{t_k} u_\matrixA)^T)^{-1} \Z^{m+n}$, for all $(\cc,\dd)\in g_{t_k} \Lambda_\matrixA$ we have $\aa\cdot\cc + \bb\cdot\dd \in \Z$. Now for each $i\in\Z$ let $H_{k,i}' = \{(\cc,\dd)\in\R^{m+n} : \aa\cdot\cc + \bb\cdot\dd = i\}$ and $H_{k,i} = \{\cc\in \R^m : (\cc,\0) \in H_{k,i}'\}$. Note that the distance between $H_{k,i}$ and $H_{k,i+1}$ is $1/\|\aa\| \geq k \to \infty$. So when $k$ is sufficiently large, Bob's ball only intersects at most \textit{one} of the hyperplane-neighborhoods $N(H_{k,i},2C e^{-\frac{t_k}{m}})$. Indeed, the distance between any two such intervals is at least $e^{-\frac{t_k}{m}}(k - 4C) \asymp e^{-\frac{t_k}{m}}k$, whereas Bob's ball is only of width at most $\frac{8Ce^{-\frac{t_k}{m}}}{\beta^2} \asymp e^{-\frac{t_k}{m}}$. Therefore Bob's choice can intersect only at most one such hyperplane-neighborhood, and Alice can delete it when appropriate. 

So it remains to show that the outcome of the game $\yy$ lies in $D$. \comkeith{If $y \notin D$ then a certain inequality has to hold for all $Q$ sufficiently large. Find $k$ such that $t_k > \log(Q_0)$, then proceed as follows. } Fix $C \geq 1$ and let $Q = e^\frac{t_k}{n}$. If there exist $\pp,\qq$ with $\|\qq\| \leq Q$ for which $Q^\frac{n}{m}\|\matrixA \qq + \pp - \yy\| \leq C$, then we will show that $\yy$ was deleted by Alice during one of her moves. Consider the point $\yy_k := g_{t_k}(\yy,0)^T = (e^\frac{t_k}{m}\yy, 0)^T$, and its distance to the lattice $\Lambda_k := g_{t_k} \Lambda_\matrixA$. 
\[
\dist(\yy_k, \Lambda_k) \leq \max\{e^\frac{t_k}{m}\|\matrixA\qq + \pp - \yy\|, e^{-\frac{t_k}{n}}\|\qq\|\} \leq \max\{C, 1\} = C.
\]
Therefore
\[
\dist(\yy_k,H_{k,i}') \leq C \text{ for some $i\in\Z$.}
\]
Say for example $\dist(\yy_k,(\cc,\dd)) \leq C$ for some $(\cc,\dd)\in H_{k,i}'$. Then $\aa\cdot\cc + \bb\cdot\dd = i$, so $\aa\cdot (\cc + \frac{\bb\cdot\dd}{\|\aa\|^2}\aa) = i$ i.e. $\cc + \frac{\bb\cdot\dd}{\|\aa\|^2}\aa \in H_{k,i}$. Since $\|\aa\| = \|\bb\|$, we have $\|\frac{\bb\cdot\dd}{\|\aa\|^2}\aa\| \leq \|\dd\| \leq C$. So $\dist(\cc,H_{k,i}) \leq C$ and thus
\[
\dist(e^\frac{t_k}{m}\yy,H_{k,i}) \leq 2C.
\]
Therefore, by the above equation, the distance between $\yy$ and $e^{-\frac{t_k}{m}}H_{k,i}$ is less than $2Ce^{-\frac{t_k}{m}}$, in which case Alice would have deleted the hyperplane-neighborhood $N(H_{k,i},Ce^{-\frac{t_k}{m}})$ containing $\yy$, a contradiction. 

To finish the proof, we recall that the hyperplane absolute game has the countable intersection property discussed above \comkeith{Discuss it above}, which shows that we can do this for every integer $C\geq 1$, and the intersection of all of these sets, which is precisely $D$, is still absolutely winning. 

Alternative reverse:
Assume there is a $C$ for every $Q$ sufficiently large, and for all $\boldsymbol\gamma$ there is a $\mathbf{q}\in\Z^n$ and an $\mathbf{p}\in\Z^m$ such that
\[
\|\mathbf{q}\|\leq Q
\]
and
\[
Q^\frac nm\|A\mathbf{q} - \mathbf{p} -\boldsymbol\gamma\| \leq C.
\]
Define
\[
t=\frac{mn}{m+n}\log\left(\frac{Q^{1+\frac{n}{m}}}{C}\right)
\]
or
\[
Q=C^\frac{m}{m+n}e^\frac{t}{n}.
\]
This implies that
\begin{align*}
\max\left(e^\frac{t}{m}\|A\mathbf{q} - \mathbf{p} -\boldsymbol\gamma\|,e^{-\frac{t}{n}}\|\mathbf{q}\|\right)
&\leq \max\left(e^\frac{t}{m}\frac{C}{Q^\frac{m}{n}},e^{-\frac{t}{n}}Q\right)\\
&= C^\frac{m}{m+n}.
\end{align*}
So that the RHS of this equation is an upper bound on the codiameter of the lattice $g_t \Lambda_A$, and thus by Lemma \ref{lemmacodiamest} we have
\[
\lambda_{m+n}\leq 2C^\frac{m}{m+n}.
\]
By Minkowski's theorem, $\frac{2^{m+n}}{(m+n)!}\leq\lambda_1\cdots\lambda_{m+n}$.
Therefore,
\[
\lambda_1\geq \frac{2^{m+n}}{(m+n)!\lambda_{m+n}^{m+n-1}}\geq \frac{2}{(m+n)!C^\frac{m(m+n-1)}{m+n}}.
\]
So that for all $t>t_0$,
\begin{equation*}
\Delta(g_t\Lambda_A)\leq \log \bigg(\frac{(m+n)!C^\frac{m(m+n-1)}{m+n}}{2}\bigg).
\end{equation*}
Thus, one has that $\Delta(g_t\Lambda_A)$ is bounded above. 
By the Dani correspondence theorem, this implies that $A$ is badly approximable.
\end{proof}
}
\comment{
\comkeith{Cut this, but possibly keep Lemma 2.4 below}

\begin{theorem}
Let $\matrixA \in \MM_{m\times n}$. Then  $\matrixA$ is not very well approximable \comdavid{define VWA} if and only if for all $\varepsilon > 0$, there  exists a constant $C$ (depending on $\varepsilon$) such that for all sufficiently large $Q$ and all $\gamma$, there exist $\pp\in\Z^m$ and $\qq\in\Z^n$ with $0 < \|\qq\| < Q$ satisfying
\[
Q^{1-\varepsilon}\ \|\matrixA\qq + \pp + \boldsymbol\gamma\| \leq C.
\]
\end{theorem}

\begin{proof}
This proof is similar to that of the previous theorem, except that the lower bound on the length of the shortest nonzero vector in the lattice $\Lambda_\matrixA$ now comes from the assumption that $\matrixA$ is not in VWA. Fix $\sigma > 2$.
For convenience, for any lattice $\Lambda$, let $\Delta(\Lambda) = \max_{\xx \ne \0} -\log(\left\| \xx \right\|)$, which captures the length of the shortest nonzero vector of the lattice. 
By \cite[Lemma 8.3, Theorem 8.5]{KleinbockMargulis}, there exists a function $r(t)$ for which 
\[
\matrixA \text{ is $\sigma$-VWA if and only if } \Delta(g_t \Lambda_\matrixA) \geq r(t) \text{ for unboundedly large $t$}.
\]
Explicitly, that function is $r(t) = \frac{\sigma-2}{\sigma} t$ in this situation. The above then says that $\matrixA$ is $\sigma$-VWA if and only if, for unboundedly large times, the lattice $g_t \Lambda_\matrixA$ has a nonzero vector of length less than $e^{-t (\sigma-2)/\sigma}$. So by our assumption that $\matrixA$ is not VWA, and hence not $\sigma$-VWA, we have that the length of the shortest nonzero vector of $g_t \Lambda_\matrixA$ is bounded from below by $e^{-t (\sigma-2)/\sigma}$ for all $t$ greater than some $T_0$.

Now, as before, this lower bound on the length of the shortest nonzero lattice vector will give rise to an upper bound on the length of the shortest nonzero grid vector for any translate of $g_t \Lambda_\matrixA$. By Minkowski's Second Theorem \comkeith{ref?}, if we denote by $\lambda_1$ the length of the shortest nonzero lattice vector and by $\lambda_2$ the length of the shortest lattice vector linearly independent from the previous vector, then 
\[
C_0 \leq \lambda_1 \lambda_2 \leq C_1
\]
for some constants $C_0, C_1$.
We know from above that $\lambda_1 \geq e^{-t(\sigma-2)/\sigma}$, so we obtain that $\lambda_2 \leq C_1 e^{t(\sigma-2)/\sigma}$. This means that the codiameter of the lattice $g_t \Lambda_\matrixA$ is bounded by some constant $C_2$ times $e^{t(\sigma-2)/\sigma}$. 
Note that, unlike in the previous case, this bound depends on $t$ and indeed increases to infinity as $t$ grows. 

Returning to the original inequality, fix $Q$ and let \[
t = \frac{\log(Q/C_2)}{1 + \frac{\sigma-2}{\sigma}}.
\]
Consider the grid $g_t \Lambda_\matrixA + \begin{pmatrix} e^{t/m} \boldsymbol\gamma \\  0 \end{pmatrix}$. As in the previous theorem, we know that there exists a grid vector whose supremum norm does not exceed $C_2 e^{t(\sigma-2)/m\sigma}$, so let this grid vector be given by the formula
\[
\begin{pmatrix} e^{t/m}&0\\0&e^{-t/n}\end{pmatrix} \begin{pmatrix} 1&\matrixA\\ 0&1\end{pmatrix} \begin{pmatrix} \pp\\\qq\end{pmatrix} + \begin{pmatrix} e^{t/m} \boldsymbol\gamma\\ 0 \end{pmatrix} = \begin{pmatrix} e^{t/m}\left(\matrixA\qq + \pp + \boldsymbol\gamma\right)\\ e^{-t/m} \qq \end{pmatrix}
\]
for some $\begin{pmatrix} \pp \\ \qq\end{pmatrix}\in \Z^{m+n}$. Our choice of $t$ makes the following true (this is what inspired the choice of $t$): 
\[
e^{t/m} = \frac{Q}{C_2 e^{t(\sigma-2)/m\sigma}}.
\]
Using this and examining the second component, we have 
\begin{align*}
e^{-t/m}\|\qq\| &\leq C_2 e^{t(\sigma-2)/m\sigma}\\
\frac{C_2 e^{t(\sigma-2)/m\sigma}}{Q}\|\qq\| &\leq C_2 e^{t(\sigma-2)/m\sigma}\\
\|\qq\|&\leq Q.
\end{align*}
For the first component, we similarly have 
\begin{align*}
e^{t/m}\left|\matrixA\qq + \pp + \boldsymbol\gamma\right| & \leq C_2 e^{t(\sigma-2)/m\sigma}\\ e^{t(1-(\sigma-2)/\sigma)/m}\left|\matrixA\qq + \pp + \boldsymbol\gamma\right| & \leq C_2\\ e^{2t/m\sigma} \left|\matrixA\qq + \pp + \boldsymbol\gamma\right| & \leq C_2\\ \left(\frac{Q}{C_2}\right)^{1/(\sigma-1)} \left|\matrixA\qq + \pp + \boldsymbol\gamma\right| & \leq C_2\\ Q^{1/(\sigma-1)}\left|\matrixA\qq + \pp + \boldsymbol\gamma\right| & \leq C.
\end{align*}

We obtain the desired bound by setting $1-\varepsilon = \frac1{\sigma-1}$, and note that for any $\varepsilon > 0$, the corresponding $\sigma$ is $>2$, and vice versa. In fact, solving gives $\varepsilon = \frac{\sigma-2}{\sigma-1}$.

\comkeith{Need to prove the converse still}

We will prove the reverse direction by contraposition, so we assume that $\matrixA \in \VWA_\sigma$ for some $\sigma > 2$. Define $\mu=1+\frac{\sigma}{2}$. Note that $2 < \mu < \sigma$.
Consider the set
\begin{equation*}
    D_{C,\mu}:= \left\{\gamma\in\R^m:\forall Q_0\in\R^+\,\exists Q>Q_0\,\exists \pp\in\Z^m,\qq\in\Z^n,\,0<\|\qq\|<Q,\,\|\matrixA\qq + \pp + \boldsymbol\gamma\|>\frac{C}{Q^\frac1{\mu-1}}\right\}.
\end{equation*}
We will show that this set is absolutely winning.  This implies that it is uncountable, and so nonempty.
To show $D_{C,\mu}$ is absolutely winning, we exhibit a winning strategy for Alice (player II) in the Absolute Winning game on $\R$ with $D_{C,\mu}$ as her target set.

The strategy is as follows. 
First, Alice seeks a sequence of times $t$ at which $e^{\frac{\sigma-2}{\sigma}t}\|g_t\Lambda_A\|$ has a local minimum.
The following lemma demonstrates that this sequence exists and occurs at the same times as the local minima of $\|g_t\Lambda_A\|$.
\comkeith{Lemma 2.4, possibly worth keeping}
}

Before proving Theorem \ref{VWA_equiv}, we recall some standard notation for viewing Diophantine approximation via homogeneous dynamics:
Consider the unimodular lattice 
\[
\Lambda_\matrixA := \begin{pmatrix} 1&\matrixA \\ 0&1\end{pmatrix} \Z^{m+n}
\]
and the flow on the space of lattices given by 
\[
g_t := \begin{pmatrix} e^{t/m} I_m &0\\0&e^{-t/n} I_n\end{pmatrix}.
\]
\begin{lemma}\label{SimulMin}
 Let $A$ be an $m\times n$ matrix that is $\psi$-approximable\footnote{See the in-line definition at the start of Section \ref{Intro_Dio}} and $r$ be the function corresponding to $\psi$ in \cite[Theorem 8.3]{KleinbockMargulis}.
There is an unbounded sequence of times $t_k$ at which $t\mapsto e^{r(t)}\|g_t\Lambda_A\|$ has a local minimum smaller than 1.
Furthermore, the set $\{t_k\}_{k\in\N}$ is exactly the set of local minima of $t\mapsto \|g_t\Lambda_A\|$.
\end{lemma}

\begin{proof}
Consider the function $\Delta(t):=-\log(\min_{\xx\in g_t\Lambda_A\setminus\{\0\}}\|\xx\|)$.
For any particular $\xx\in \Lambda_A$, the function $t\mapsto \|g_t \xx\|$ has an interval on which it is proportional to $e^\frac{t}{m}$ and an interval on which it is proportional to $e^{-\frac{t}{n}}$, and the intersection point $t(\xx)$ of these intervals is a unique minimum for $t\mapsto \|g_t \xx\|$.
If $\xx\in\Lambda_A$ is fixed, then the set $\{t\geq 0: \|g_t\xx\|=e^{-\Delta(t)}\}$ is an interval containing $t(\xx)$. This is due to the fact that all vectors with lengths that are contracting have lengths proportional to $e^{-\frac{t}{n}}$, and all vectors with lengths that are expanding have lengths proportional to $e^{\frac{t}{m}}$. Since $g_t\Lambda_A$ is a lattice, the set of lengths of its vectors is discrete.
Therefore, $\Delta$ is piecewise linear, with regions of slope $-\frac{1}{m}$ alternating with regions of slope $\frac{1}{n}$.
By \cite[Theorem 8.5]{KleinbockMargulis}, there are arbitrarily large positive times $t_\ell$ at which $\Delta(t_\ell)\geq r(t_\ell)$.
Given one of these times, $t_\ell$, if it lies in an interval where $\Delta(t)=\Delta_0-\frac{t}{m}$, one may replace $t_\ell$ with the time of the previous maximum, say $t_k$.
This is because
\begin{equation*}
    m\Delta(t_k)+t_k=m\Delta_0=m\Delta(t_\ell)+t_\ell\geq mr(t_\ell)+t_\ell\geq mr(t_k)+t_k
\end{equation*}
(for the last inequality see \cite[Eq. (8.2b)]{KleinbockMargulis}).
Similarly, if $t_\ell$ lies in a region where $\Delta(t)=\Delta_0+\frac{t}{n}$, one may replace $t_\ell$ with the time of the next maximum, say $t_{k+1}$, since
\begin{equation*}
    n\Delta(t_{k+1})-t_{k+1}=n\Delta_0=n\Delta(t_\ell)-t_\ell\geq nr(t_\ell)-t_\ell\geq nr(t_{k+1})-t_{k+1}
\end{equation*}
(here the last inequality follows from \cite[Eq. (8.2a)]{KleinbockMargulis}).
So, the times at which $\Delta$ has a maximum are also where $r-\Delta$ has a minimum.
Exponentiating, one has that $\{t_k\}_{k\in\N}$ is an unbounded sequence of local minima of $t\mapsto e^{r(t)}\|g_{t}\Lambda_A\|$.
\end{proof}

In the proof of the theorem below, we will need to prove that a certain set is nonempty. We will do this by establishing something much stronger, namely that is a winning set for a modification of a game introduced by Schmidt in \cite{Schmidt1}. This modification, known as the \textit{hyperplane absolute game} \cite{BFKRW}, is played as follows (see the discussion in \cite{BFKRW} for a nice overview of the development of various modifications of Schmidt's original game). Fix a parameter $0 < \beta < \frac13$. Bob, the first player, chooses a point $\xx_1\in \R^m$ and a radius $\rho_1$, thus defining a closed ball $B(\xx_1, \rho_1) \subset \R^m$. Alice, the second player, then chooses an $(m-1)$-dimensional affine hyperplane $\LL_1$ and deletes an $\varepsilon_1\rho_1$-neighborhood of $\LL_1$ from Bob's ball $B(\xx_1,\rho_1)$, where $0<\varepsilon_1<\beta$. For convenience, we denote the $\varepsilon_1\rho_1$-neighborhood of $\LL_1$ by $A_1$. Bob then picks a point $\xx_2$ and a radius $\rho_2 \geq \beta \rho_1$ satisfying 
\[
B(\xx_2, \rho_2) \subset B(\xx_1, \rho_1) \setminus A_1.
\] Alice then chooses a new affine subspace $\LL_2$, and the game proceeds as before. A set $\SS$ is said to be $\beta$-\textit{hyperplane absolute winning}, written HAW, if Alice has a strategy such that, no matter how Bob plays, she can guarantee that the intersection point $\xx_\infty$ defined by 
\[
\{\xx_\infty\} := \bigcap_i B(\xx_i, \rho_i) \textup{ satisfies } \xx_\infty\in \SS.
\]
We need two important facts about hyperplane absolute winning sets, both proven in \cite[Prop. 2.3]{BFKRW}: (1) HAW sets have full Hausdorff dimension; and (2) the countable intersection of HAW sets is again HAW. 

\ignore{
By the lemma, the sequence of times Alice seeks exists.
Passing to a subsequence if necessary, Alice ensures that $t_{k+1}-t_k+r(t_k)-r(t_{k+1})>-2\log(\beta)$.
For each $t_k$, Alice chooses a vector of minimal length, $\rr_k\in g_{t_k}\Lambda_A$.
By construction, this vector lies along the line $\{y=x\}$ (or $\{y=-x\}$), and so is linearly independent of a vector $(y\,0)^T$ along the $x$-axis.
Let $w_k$ be the value of $y$ that makes $\mathbf{w}_k=(w_k\,0)^T$ such that $\det|\mathbf{r}_k\mathbf{w}_k|=1$.
\begin{equation*}
    |w_k|=\|\mathbf{w}_k\|\geq\frac{1}{\|\mathbf{r}_k\|}\geq e^{r(t_k)}=e^{\frac{\sigma-2}{\sigma}t_k}.
\end{equation*}
When Bob's ball has radius between $\beta^{-1}e^{-\frac{2}{\sigma}t_k}$ and $\beta^{-2}e^{-\frac{2}{\sigma}t_k}$,
Alice deletes an interval of radius $2C^{1-\frac{1}{\mu}}e^{-\frac{2}{\mu}t_k}$ around the integer multiple of $e^{-t_k}w_k$ closest to the center of Bob's interval (if there are two such points, she chooses the leftmost).
If Bob's ball has radius larger than the prescribed interval, Alice finds the smallest $k'\geq k$ such that the interval of radius $2C^{1-\frac{1}{\mu}}e^{-\frac{2}{\mu}t_{k'}}$ around a multiple of $e^{-t_{k'}} w_{k'}$ includes the center of Bob's ball and deletes this interval (if no such interval exists, she chooses an interval that does not intersect Bob's ball).
This concludes Alice's strategy.
The inequality $\mu<\sigma$ implies $e^{-\frac{2}{\mu}t_k}<e^{-\frac{2}{\sigma}t_k}$, so that the intervals Alice wishes to delete shrink faster than the distance between them.
The inequality $t_{k+1}-r(t_{k+1})-(t_k-r(t_k))=\frac{2}{\sigma}(t_{k+1}-t_k)>-2\log(\beta)$ ensures that Bob's ball will eventually be smaller than the distance between the intervals Alice considers for deletion at stage $k$.
So, eventually, there will be at most one interval that Alice wishes to delete that intersects Bob's ball.

It remains to show that Alice's strategy is a winning strategy.
Suppose $y\notin D_{C,\mu}$
At some sufficiently large stage $k$, one has that for $\mu>2$ there are $m,n\in\Z$ with $|n|<Q_k=C^{1-\frac{1}{\mu}}e^{\left(2-\frac{2}{\mu}\right)t_k}$ such that $|\matrixA n-m-y|< C Q_k^{-\frac{1}{\mu-1}}$.
Consider the point $\yy_k := g_{t_k}(y,0)^T = (e^{t_k}y, 0)^T$, and its distance to the lattice $\Lambda_k := g_{t_k} \Lambda_\matrixA$. 
\[
\dist(\yy_k, \Lambda_k) \leq \max\{e^{t_k}|n\matrixA-m-y|, e^{-t_k}|n|\} <\max\{e^{t_k}|n\matrixA-m-y|, e^{-t_k}C^{1-\frac{1}{\mu}}e^{\left(2-\frac{2}{\mu}\right)t_k}\} \leq C^{1-\frac{1}{\mu}}e^{\frac{\mu-2}{\mu}t_k}.
\]
Therefore
\[
\dist(\yy_k, \Lambda_k + \{y=x\})= \dist(\yy_k, \Z (w_k,0) + \{y=x\}) \leq C^{1-\frac{1}{\mu}}e^{\frac{\mu-2}{\mu}t_k}.
\]
By definition of the sup norm, there exists some integer $p$ and a point $(s,s)$ on the line $\{y=x\}$ for which 
\[
\max\{|e^{t_k}y - p w_k - s|,|s| \}\leq C^{1-\frac{1}{\mu}}e^{\frac{\mu-2}{\mu}t_k}.
\]
Therefore, by the above equation, the distance between $y$ and $e^{-t_k} p w_k$ is less than $2C^{1-\frac{1}{\mu}}e^{-\frac{2}{\mu}t_k}$, in which case Alice would have deleted the interval containing $y$, a contradiction.
Thus, the intersection of Bob's balls must lie in $D_C$, and Alice's strategy is a winning strategy.
Since varying $C$ or $\mu$ yield nested sets, one may reduce an intersection over these parameters to a countable one.
As the set $D_{C,\mu}$ is absolute winning, so is a countable intersection of such sets.
So,
\begin{equation*}
    D:=\left\{y\in\R:\forall\epsilon>0,\exists C>0,\forall Q_0\in\R^+\exists Q>Q_0\exists n,m\in\Z,|n|<Q,|\matrixA n-m-y|>\frac{C}{Q^{1-\epsilon}}\right\}
\end{equation*}
is absolute winning, and so nonempty, as was to be shown.
\end{proof}

}

\begin{theorem}\label{VWA_equiv}
Let $A \in \mathcal{M}_{m\times n}(\R)$. Then  $A$ is not very well approximable  if and only if for all $\varepsilon > 0$, there is a $c>0$, such that for all sufficiently large $Q$ and all $\bfgamma$, there exist $\qq\in\Z^n$, $\|\qq\| < Q$ and $\pp\in \Z^m$ satisfying
\begin{equation}\label{VWA_equiv_eq}
\|A\qq - \pp - \boldsymbol\gamma\| < \frac{c}{Q^{\frac{n}{m}-\epsilon}}.
\end{equation}
\end{theorem}
\begin{proof}
Since $A$ is not very well approximable, neither is its transpose $A^T$.
As before, we define $A^T$ to be \emph{$\nu$-VWA} if the Diophantine exponent $\sigma\left(A^T\right)$ satisfies $\sigma\left(A^T\right)>\nu>\frac mn$ (this is the same as in Definition \ref{BA_VWA}, but the dimensions are reversed).
Fix $\sigma>\frac{m}{n}$. As $A^T$ is not $\sigma-$VWA, 
 it is in particular not $\sigma$-approximable, and therefore by \cite[Theorem 8.3]{KleinbockMargulis} we have $\Delta((g_t\Lambda_A)^*)<r(t)$ for all $t$ sufficiently large, where $r(t)=\theta_\sigma t := \frac{n\sigma-m}{(\sigma+1)mn}\ t$ and $\Lambda^* = \{\ss\in\R^d : \rr\cdot\ss \in \Z \;\forall \rr\in \Lambda\}$ is the polar body of $\Lambda$. By the definition of $\Delta$ given in Lemma \ref{SimulMin} above, this means that $\min_{\xx\in (g_t\Lambda_A)^* \butnot \{\0\}}(e^{r(t)}\|\xx\|)>1$, i.e. that $\lambda_1((g_t\Lambda_A)^*)>e^{-r(t)}$, where $\lambda_i(\Lambda)$ denotes the $i$th Minkowski minimum of the lattice $\Lambda$.
By a theorem of Mahler \cite[end of section 3]{Mahler_Transference} $\lambda_{m+n}(g_t\Lambda_A)\asymp\frac{1}{\lambda_1((g_t\Lambda_A)^*)}$.
This gives us an upper bound: $\lambda_{m+n}(g_t\Lambda_A)\asymp\frac{1}{\lambda_1((g_t\Lambda_A)^*)}<e^{r(t)}=e^{\theta_\sigma t}$, and thus on the codiameter of $g_t \Lambda_A$ (as the codiameter is $\asymp$ to $\lambda_{m+n}$).  That is, there is a uniform constant $c_2>0$ such that the codiameter is no larger than $ c_2 e^{\theta_\sigma t}$. Then there exists $(\pp,\qq)\in\Z^{m+n}$ such that
\[
\|g_t(A\qq - \pp,\qq) - g_t(\boldsymbol\gamma,\mathbf 0)\| < c_2 e^{\theta_\sigma t}.
\]
Fix $Q$ large enough, and set
\begin{equation*}
    t=\frac{1}{\frac{1}{n}+\theta_\sigma }\log\left(\frac{Q}{c_2}\right).
\end{equation*}
Then one has that
\begin{align*}
    e^{-\frac{t}{n}}\|\qq\|&<c_2e^{\theta_\sigma t}\nonumber\\
    \|\qq\|&<c_2e^{\left[\theta_\sigma +\frac{1}{n}\right]t}\nonumber\\
    \|\qq\|&<Q.
\end{align*}
Similarly,
\begin{align*}
    e^{\frac{t}{m}}\|A\qq-\pp-\boldsymbol\gamma\|&<c_2e^{\theta_\sigma t}\nonumber\\
    \|A\qq-\pp-\boldsymbol\gamma\|&<c_2e^{\left[\theta_\sigma -\frac{1}{m}\right]t}\nonumber\\
    \|A\qq-\pp-\boldsymbol\gamma\|&<cQ^{\left[\theta_\sigma -\frac{1}{m}\right]/\left[\frac{1}{n}+\theta_\sigma \right]}\nonumber\\
    \|A\qq-\pp-\boldsymbol\gamma\|&<\frac{c}{Q^{\frac{n}{m}-\varepsilon}},
\end{align*}
where $\varepsilon=\frac nm-\frac 1\sigma$ and $c=c_2^{1+\frac nm-\varepsilon}$.

We will prove the reverse direction by contraposition.  Assume that $\matrixA^T$ is $\sigma$-VWA for some $\sigma > \frac{m}{n}$.  Define $\mu=\frac{1}{2}(\frac{m}{n}+\sigma)$, $\varepsilon(\mu)=\frac nm-\frac1\mu$, $\theta_\sigma=\frac{n\sigma-m}{nm(\sigma+1)}$, and $\theta_\mu=\frac{n\mu-m}{nm(\mu+1)}$. Note that $\frac mn<\mu<\sigma$, $\epsilon > 0$, and $0<\theta_\mu<\theta_\sigma$.
Consider the set
\begin{equation*}
  D_{c,\mu}:=\left\{\bfgamma\in\R^m:\forall Q_0>0\;\exists Q>Q_0\;\forall\pp\in\Z^m\;\qq\in\Z^n,\,\|A\qq-\pp-\bfgamma\|>\frac{c}{Q^{\frac{n}{m}-\epsilon(\mu)}}\right\}.
\end{equation*}
We will show that this set is hyperplane absolute winning.  

To show $D_{c,\mu}$ is hyperplane absolute winning, we exhibit a winning strategy for Alice (player II) in the hyperplane absolute game on $\R^m$ with $D_{c,\mu}$ as her target set.
Alice begins by finding the sequence of times $\{t_k\}_{k\in\N}$ at which $\|(g_t\Lambda_A)^*\|$ and $e^{\theta_\sigma t}\|(g_t\Lambda_A)^*\|$ are simultaneously locally minimized, which is possible by Lemma \ref{SimulMin}. 
Passing to a subsequence if necessary, Alice ensures that $\left(\theta_\mu-\frac{1}{m}\right)(t_{k+1}-t_k)=-\frac{n+m}{nm(\mu+1)}\left(t_{k+1}-t_k\right)<2\log(\beta)$. 
For each $k$, Alice chooses a vector in $(g_{t_k} \Lambda_A)^* \butnot\{\0\}$ of minimal length, say
$\mathbf{r}_k=(\mathbf{a}_k,\mathbf{b}_k)$.  Since $t_k$ is a local minimum of $t\mapsto e^{r(t)} \|g_t \Lambda_A\|$, we have $\|\mathbf{r}_k\|=\|\mathbf{a}_k\|=\|\mathbf{b}_k\|<e^{-\theta_\sigma t_k}$. 

Recall that $(g_{t_k} \Lambda_\matrixA)^* = ((g_{t_k} u_\matrixA)^T)^{-1} \Z^{m+n}$. Thus for all $(\cc,\dd)\in g_{t_k} \Lambda_\matrixA$ we have $\aa_k\cdot\cc + \bb_k\cdot\dd \in \Z$. Now for each $i\in\Z$ let $H_{k,i}' = \{(\cc,\dd)\in\R^{m+n} : \aa_k\cdot\cc + \bb_k\cdot\dd = i\}$ and $H_{k,i} = \{\cc\in \R^m : (\cc,\0) \in H_{k,i}'\}$. Note that the distance between $H_{k,i}$ and $H_{k,i+1}$ is $1/\|\aa_k\|>e^{\theta_\sigma t_k}>e^{\theta_\mu t_k}$ 

Let $\theta_\mu = \frac{n\mu - m}{mn(\mu+1)} = \frac{\varepsilon}{n(1 + (n/m) - \varepsilon)} > 0$. When Bob's ball has radius between $2c\exp\left(-\left(\frac1m - \theta_\mu\right)t_k\right)\beta^{-1}$ and $2c\exp\left(-\left(\frac1m - \theta_\mu\right)t_k\right)\beta^{-2}$,
Alice deletes a hyperplane neighborhood of radius $2c\exp\left(-\left(\frac1m - \theta_\mu\right)t_k\right)$ around $e^{-\frac{t_k}{m}}H_{k,i}$, where $i$ is chosen so that $H_{k,i}$ is the closest hyperplane to the center of Bob's ball (if there are two such hyperplanes, she chooses the smaller index).
If Bob's ball has radius larger than the prescribed interval, Alice finds the smallest $k'\geq k$ such that the hyperplane neighborhood of radius $2c\exp\left(-\left(\frac1m - \theta_\mu\right)t_k\right)$ around a multiple of $e^{-\frac{t_k}{m}}H_{k,i}$ includes the center of Bob's ball and deletes this hyperplane neighborhood (if no such hyperplane neighborhood exists, she chooses a hyperplane neighborhood that does not intersect Bob's ball).
This concludes Alice's strategy.
Since $\exp(\theta_\mu t_k)<\exp\left(\theta_\sigma t_k\right)<\text{dist}(H_{k,i},H_{k,i+1})$ the hyperplane neighborhoods Alice wishes to delete shrink faster than the distance between them.
The inequality $-\frac{n+m}{nm(\mu+1)}\left(t_{k+1}-t_k\right)<2\log(\beta)$ ensures that Bob's ball will eventually be smaller than the distance between the hyperplane neighborhoods Alice considers for deletion at stage $k$ (and that it need not shrink too quickly).
So, eventually, there will be at most one hyperplane neighborhood that Alice wishes to delete that intersects Bob's ball.
It remains to show that Alice's strategy is a winning strategy.
Suppose $\bfgamma\notin D_{c,\mu}$.
At some sufficiently large stage $k$, one has that for $\mu>\frac{m}{n}$ there are $\mathbf{q}_k\in\Z^n$, $\mathbf{p}_k\in\Z^m$ with $\|\qq_k\|<Q_k=ce^{(\theta_\mu+\frac1n)t_k}$ such that $\|A \qq_k - \pp_k - \bfgamma\|< c Q_k^{-(\frac{n}{m}-\epsilon(\mu))}$. Consider the point $\yy_k := g_{t_k}(\zz_k,0)^T = (e^\frac{t_k}{m}\zz_k, 0)^T$, and its distance to the lattice $\Lambda_k := g_{t_k} \Lambda_\matrixA$. 
\[
\dist(\yy_k, \Lambda_k) \leq \max\{e^\frac{t_k}{m}\|A\qq_k-\pp_k-(\zz_k,\0)\|, e^{-\frac{t_k}{n}}\|\qq_k\|\} < c\exp(\theta_\mu t_k).
\]
Therefore
\[
\dist(\yy_k,H_{k,i}') < c\exp(\theta_\mu t_k)\text{ for some $i\in\Z$.}
\]
Say for example $\dist(\yy_k,(\cc,\dd)) < c\exp(\theta_\mu t_k)$ for some $(\cc,\dd)\in H_{k,i}'$. Then $\aa_k\cdot\cc + \bb_k\cdot\dd = i$, so $\aa_k\cdot (\cc + \frac{\bb_k\cdot\dd}{\|\aa_k\|^2}\aa_k) = i$, i.e., $\cc + \frac{\bb_k\cdot\dd}{\|\aa_k\|^2}\aa_k \in H_{k,i}$. Since $\|\aa_k\| = \|\bb_k\|$, we have $\|\frac{\bb_k\cdot\dd}{\|\aa_k\|^2}\aa_k\| \leq \|\dd\| < c\exp(\theta_\mu t_k)$. So $\dist(\cc,H_{k,i}) < c\exp(\theta_\mu t_k)$ and thus
\[
\dist(e^{t_k/m}\zz_k,H_{k,i}) < 2c\exp(\theta_\mu t_k).
\]
Therefore, by the above equation, the distance between $\zz_k$ and $e^{-t_k/m}H_{k,i}$ is less than $2c\exp\left(\left(\theta_\mu-\frac{1}{m}\right)t_k\right)$, in which case Alice would have deleted the hyperplane-neighborhood $N\left(e^{-t_k/m}H_{k,i},2c\exp\left(\left(\theta_\mu-\frac{1}{m}\right)t_k\right)\right)$ containing $\zz_k$, a contradiction.
Thus, the intersection of Bob's balls must lie in $D_{c,\mu}$, and Alice's strategy is a winning strategy. So $D_{c,\mu}$ is HAW. 
Since we can write 
\begin{align*}
    D &:=\Big\{\zz\in\R^m:\forall\epsilon>0\;\forall C>0\;\forall Q_0\in\R^+\;\exists Q>Q_0\;\exists \qq,m\in\Z,|n|<Q,|A \qq-\pp-\zz|>\frac{C}{Q^{\frac{n}{m}-\epsilon}}\Big\}\\ &= \bigcap_{c \in \Q} \bigcap_{\mu_j \searrow \frac{m}{n}} D_{c,\mu},
\end{align*}
where we choose a countable sequence $\mu_j$ which tends to $\frac{m}{n}$, we see that $D$
is also HAW, and so nonempty, as was to be shown.
\end{proof}

\ignore{
\begin{proof}
Alternative reverse:
Assume for every $\epsilon>0$, for every $Q$ sufficiently large, and all $\boldsymbol\gamma$ there is an$\qq\in\Z^n$ and an $\pp\in\Z^m$ such that
\[
\|\qq\|<Q-1
\]
and
\[
\|A\qq - \pp -\boldsymbol\gamma\| < \frac{2^\frac{(m+n)^2}{m}}{Q^{\frac{n}{m}-\epsilon}}.
\]
Define
\[
t=\frac{mn}{m+n}\log\left(\frac{Q^{1+\frac{n}{m}-\epsilon}}{2^\frac{(m+n)^2}{m}}\right)
\]
or
\[
Q=\left(2^\frac{(m+n)^2}{m}e^{\left(\frac{1}{m}+\frac{1}{n}\right)t}\right)^\frac{1}{1+\frac{n}{m}-\epsilon}.
\]
This implies that
\begin{align*}
\max\left(e^\frac{t}{m}\|A\qq - \pp -\boldsymbol\gamma\|,e^{-\frac{t}{n}}\|\qq  + \mathbf\matrixA\|\right)
&\leq \max\left(e^\frac{t}{m}\frac{2^\frac{(m+n)^2}{m}}{Q^{\frac{n}{m}-\epsilon}},e^{-\frac{t}{n}}Q\right)\\
&= 2^\frac{m+n}{1-\frac{m}{m+n}\epsilon}e^\frac{\epsilon t}{n(1+\frac{n}{m}-\epsilon)}.
\end{align*}
So that the RHS of this equation is an upper bound on the codiameter of the lattice $g_t \Lambda_A$, and thus by Lemma \ref{lemmacodiamest} we have
\[
\lambda_{m+n}\leq 2\cdot 2^\frac{m+n}{1-\frac{m}{m+n}\epsilon}e^\frac{\epsilon t}{n(1+\frac{n}{m}-\epsilon)}.
\]
By Minkowski's theorem, $\frac{2^{m+n}}{(m+n)!}\leq\lambda_1\cdots\lambda_{m+n}$.
Therefore,
\[
\lambda_1\geq \frac{2^{m+n}}{(m+n)!\lambda_{m+n}^{m+n-1}}\geq \frac{1}{2^{d-1} (m+n)!}2^{-(m+n)\left(\frac{m+n-1}{1-\frac{m}{m+n}\epsilon}\right)}e^{-\frac{\epsilon (m+n-1) t}{n(1+\frac{n}{m}-\epsilon)}}.
\]
So that for all $t>t_0(\epsilon)$,
\[
\Delta(g_t\Lambda_A)\leq \frac{\epsilon (m+n-1) t}{n(1+\frac{n}{m}-\epsilon)} +(m+n)\left(\frac{m+n-1}{1-\frac{m}{m+n}\epsilon}\right)\log(2)+\log(2^{d-1} (m+n)!).
\]
Thus, one has that
\[
\limsup_{t\to\infty}\frac{\Delta(g_t\Lambda_A)}{t}=0.
\]
By the Dani correspondence theorem, this implies that $A$ is not very well approximable.
\end{proof}

\begin{lemma}
\label{lemmacodiamest}
Let $\Lambda$ be a lattice in $\R^d$. Then the codiameter of $\Lambda$ is $\geq (1/2) \lambda_d(\Lambda)$.
\end{lemma}
\begin{proof}
Let $\rr_1,\ldots,\rr_{d-1}\in\Lambda$ be a linearly independent sequence so that $\|\rr_i\| \leq \lambda_i(\Lambda)$ for all $i$, and let $\rr\in\Lambda$ be a representative of a generator of $\Lambda / \sum_1^{d-1} \R\rr_i$. Then by definition, the codiameter of $\Lambda$ is $\geq \dist(\rr/2,\Lambda)$. Fix $\rr'\in\Lambda$ such that $\dist(\rr/2,\Lambda) = \|\rr/2 - \rr'\| = (1/2) \|\rr-2\rr'\|$. Now $\rr-2\rr'\in\Lambda\butnot \sum_1^{d-1} \R\rr_i$; it follows from the definition of $\lambda_d(\Lambda)$ that $\|\rr-2\rr'\|\geq \lambda_d(\Lambda)$.
\end{proof}
}

\begin{remark}
As stated, one can recover both Theorems \ref{BA_equiv} and \ref{VWA_equiv} from \cite{MoshchevitinNeckrasov}. In the case of Theorem \ref{VWA_equiv}, applying their Theorem 1A with $f(T) = \frac1{T^{1 + \varepsilon}}$, establishes that if a matrix is not VWA, then the set of $\gamma$ which satisfy Equation (\ref{VWA_equiv_eq}) is all of $\R^m$. For the converse direction, it is shown in their Theorem 1B, using the same function $f$, that if the matrix were VWA, then the set of $\gamma$ for which Equation (\ref{VWA_equiv_eq}) does \textit{not} hold is of full measure. For the proof of Theorem \ref{VWA_equiv}, it is enough that when $A$ is VWA that this set is simply nonempty, although our proof shows that this set is hyperplane absolute winning.
\end{remark}

\subsection{Proof of Theorem \ref{theorem_main}}\label{S:proof}
Both of the above theorems are strictly about properties of a matrix and solutions to some inhomogeneous approximation problem in $\R^m$. We now need to connect this back to a statement about approximation on the abelian variety itself. This section is largely taken from \cite{Waldschmidt}, but we reproduce the pertinent facts here for ease of the reader. 

Let $\mathcal{A}$ be an abelian variety defined over a number field $K \subset \R$, and let $T_\mathcal{A}(\R)$ denote the tangent space at the identity. In this case, $\mathcal{A}(\R)$ is a real lie group, and we let 
\[
\exp_\mathcal{A}: T_\mathcal{A}(\R) \to \mathcal{A}(\R)
\]
denote the usual exponential map. The image of this map is precisely $\mathcal{A}(\R)^0$, and the kernel $\Omega_{\mathcal{A}, \R}$ is a discrete subgroup of $T_\mathcal{A}(\R)$ of rank $g$ over $\Z$. We fix a $\Z$-basis $\boldsymbol\gamma_1, ..., \boldsymbol\gamma_g$ of $\Omega_{\mathcal{A},\R}$. 

For each $1\leq j \leq r$, let $\boldsymbol\alpha_j \in T_\mathcal{A}(\R)$ satisfy $\exp_\mathcal{A}(\boldsymbol\alpha_j) = \beta_j$, where $\{\beta_1, ....,\beta_r\}$ denote a (maximal) $\Z$-linearly independent set in $\mathcal{A}(K) \cap \mathcal{A}(\R)^0$. 

The following result connects distances on the abelian variety to distances in $\R^g$: 
\begin{lemma}[cf. {\cite[Lemma 5.1]{Waldschmidt}}]\label{Wald_equiv} 
Let $\theta$ be a positive number and let $\Gamma = \Z \beta_1 + \Z\beta_2 + ... + \Z \beta_r \subset \mathcal{A}(K) \cap \mathcal{A}(\R)^0$. The the following are equivalent: 
\begin{enumerate}
\item\label{Wald_variety} There exists a constant $C> 0$ such that, for any $Q \geq 1$ and any $\zeta \in \mathcal{A}(\R)^0$, there exists $\beta \in \Gamma$ with $\hat{h}(\beta) < Q$ and 
\[
\dist(\beta, \zeta) \leq C Q^{-\theta/2}.
\]
\item There exists a constant $C>0$ such that for any positive integer $Q$ and any $\boldsymbol\xi \in T_\mathcal{A}(\R)$, there exist rational integers $q_1,....,q_r$ and $p_1,...,p_g$ satisfying 
\begin{equation}\label{Wald_eq}
    \displaystyle\max_{1 \leq j \leq r}|q_j| \leq Q, \quad \text{and} \quad \left\| q_1 \boldsymbol\alpha_1 + ... + q_r \boldsymbol\alpha_r + p_1 \boldsymbol\gamma_1 + ... + p_g \boldsymbol\gamma_g + \boldsymbol\xi\right\| \leq C Q^{-\theta}.
\end{equation}
\end{enumerate}
\end{lemma}

Waldschmidt's conjecture \ref{Wald_conj} is that the statement (\ref{Wald_variety}) holds with $\theta = r/g - \varepsilon$ for any simple abelian variety defined over a real number field $K$ and any $\varepsilon > 0$. We can rewrite the equation in (\ref{Wald_eq}) in the following manner. Let the matrix $H \in M_{g \times r}(\R)$ have columns given by $\{\boldsymbol\alpha_1, ...., \boldsymbol\alpha_r\}$ and the matrix $J \in M_{g\times g}(\R)$ have columns given by $\{\boldsymbol\gamma_1, ...., \boldsymbol\gamma_g\}$. Then we can write \eqref{Wald_eq} as:
\[
\left\| H \textbf{q} + J \textbf{p} + \boldsymbol\xi\right\| \leq C Q^{-r/g + \varepsilon},
\]
where $\textbf{q}$ is the column vector of the numbers $q_1,...,q_r$ and $\textbf{p}$ is the column vector of the numbers $p_1,...,p_g$. Multiplying the lefthand side by the inverse of the matrix $J$ on the left, which will change the norm by only a bounded amount, we see that 
\[
\left\| J^{-1} H \textbf{q} + \textbf{p} + J^{-1} \boldsymbol\xi\right\| \leq C Q^{-r/g + \varepsilon}.
\]
This is the same statement as in Theorem \ref{VWA_equiv} with 
\[
m = g,\ n = r,\ A = J^{-1}H, \text{ and }\boldsymbol\gamma = J^{-1}\boldsymbol\xi.
\]
Combining Lemma \ref{Wald_equiv} with Theorem \ref{VWA_equiv} shows that Waldschmidt's conjecture is true of some simple abelian variety $\mathcal{A}$ defined over $K$ if and only it the matrix $J^{-1}H$ as constructed above is not VWA. 

This finishes the proof of Theorem \ref{theorem_main}.

\begin{lemma}\label{exp_UB}
    For any simple abelian variety $\mathcal{A}$ of dimension $g$ and any number field $K$ as above, we have 
    \[
    \sigma_{\mathcal{A},K} \leq \frac{r}{2g},
    \] where, as usual, $r$ denotes the rank of the Mordell-Weil group $\mathcal{A}(K)$.
\end{lemma}
\begin{proof}
    Suppose, for contradiction, that $\sigma_{\mathcal{A},K} > \frac{r}{2g}$. Then there exists $\mu > \frac{r}{g}$ such that for every $P \in \mathcal{A}(\R)^0$, there exists a constant $C$ and infinitely many $Q \in \mathcal{A}(K)$ satisfying 
    \[
    \dist(P,Q) < C \left(\hat{h}(Q)\right)^{-\mu/2}.
    \]
    As in the proof of Lemma \ref{Wald_equiv} in \cite{Waldschmidt}, and using the notation we introduced above, if we write $Q = q_1 \beta_1 + ... + q_r \beta_r$, then we have that 
    \[
     \left\|  q_1 \boldsymbol\alpha_1 + ... + q_r \boldsymbol\alpha_r + p_1 \boldsymbol\gamma_1 + ... + p_g \boldsymbol\gamma_g + \boldsymbol\xi\right\| \asymp \dist(P,Q) < C \left(\hat{h}(Q)\right)^{-\mu/2}
    \]
    for some choice of rational integers $p_1, ...., p_g$ and where $\exp_\mathcal{A}\boldsymbol\xi = P$. Rewriting this as before, this implies

    \[
\left\| J^{-1} H \textbf{q} + \textbf{p} + J^{-1} \boldsymbol\xi\right\| \leq C' \left(\hat{h}(Q)\right)^{-\mu/2}.
\]
Because the height function $\hat{h}$ defines a positive definite quadratic form on $\R \otimes \mathcal{A}(K),$ we have, up to constant, that 
\[
\hat{h}(Q) = \hat{h}\left(q_1\beta_1 + ... + q_r\beta_r\right)  \asymp \max_{1 \leq i \leq r}|q_i|^2.
\]
Therefore, for some possibly new constant $C'$, we have that 

\begin{equation}\label{UB}
\left\| J^{-1} H \textbf{q} + \textbf{p} + J^{-1} \boldsymbol\xi\right\| \leq C' \max_{1 \leq i \leq r} |q_i|^{-\mu},
\end{equation}
for all points $J^{-1}\boldsymbol\xi\in \R^g$. However, this is never the case:
\begin{theorem*}[{\cite[Theorem 1]{BHKV}}]
    For any matrix $A \in \mathcal{M}_{g \times r}(\R),$ the Hausdorff dimension of the set 
    \[
    \left\{\xx \in \R^g :\ \exists c>0  \text{ s.t. } \left\|A \qq + \pp + \xx\right\| \geq \frac{c}{\max_i |q_i|^{r/g}}\ \forall \qq\in \Z^r,\ \forall \pp \in \Z^g\right\}
    \] equals $g$. In particular, this set is always nonempty. 
\end{theorem*}
Letting $A = J^{-1}H$ and $\xx = J^{-1}\boldsymbol\xi$ as usual, and recalling that $\mu > \frac{r}{g}$, shows that this result is incompatible with equation \eqref{UB} holding for all points $\xx \in \R^g$. Therefore, $\sigma_{\mathcal{A},K} \leq \frac{r}{2g},$ as needed. 

\end{proof}

As a consequence of Lemma \ref{exp_UB}, our conjecture \ref{Conj_FLMS} is weaker than Waldschmidt's conjecture \ref{Wald_conj}, in that a proof of \ref{Wald_conj} would immediately imply that \ref{Conj_FLMS} is true as well. 


\section{Rank 1 Elliptic Curves}

In this section, we establish further results in the case of a rank 1 elliptic curve defined over $K$. 

Let $d_E(\cdot,\cdot)$ denote a Riemannian metric on $E(\R)^0$ and let $\hat{h}$ be the N\'eron--Tate height defined on $E(K)$ associated with $d_E(\cdot,\cdot)$ (see, for instance, \cite[Chapter 8, Section 9]{Silverman}). 

\ignore{
While we have not found the Diophantine exponent of an elliptic curve to be discussed in the literature, Waldschmidt \cite{Waldschmidt} has a conjecture about abelian varieties which involves a closely related quantity. In particular, the following conjecture would be implied by \cite[Conjecture 1.1]{Waldschmidt}:
\begin{conjecture}\label{DEconj}
Let $E$ be an elliptic curve defined over $\Q$ and $r := \operatorname{rank}(E(\Q))\geq 1$. Then its Diophantine exponent satisfies:
$$\sigma_E \geq \frac{r}{2}.$$
\end{conjecture}
}
For rank 1 elliptic curves defined over $K$, we can actually establish Conjecture \ref{Conj_FLMS}:
\begin{theorem}\label{weakDT}

If the rank of $E$ over $K$ is exactly 1, then we have equality: 
\[
\sigma_{E,K} = \frac{1}{2}.
\]
Explicitly, there exists a constant $C$ depending only on the curve $E$ such that for every $P \in E(\R)^0 \setminus E(K)$, there exist infinitely many points $Q \in E(K)$ satisfying 
\begin{equation*}\label{dirichlet}
    d_E(P,Q) < C\left(\hat{h}(Q) \right)^{-1/2}.
\end{equation*}
Moreover, the function $C\left(\hat{h}(Q) \right)^{-1/2}$ cannot be improved (in a sense defined below) if the inequality in (\ref{dirichlet}) is to remain true for all points $P \in E(\R)^0\setminus E(K)$. 
\end{theorem}
In fact, the proof shows that for any elliptic curve $E$ defined over $K$ with rank \textit{at least} 1, we have 
\[
\sigma_{E,K} \geq \frac1{2},
\]
with equality in the case that the rank is exactly 1.

Before we prove Theorem \ref{weakDT}, we revisit the general discussion of abelian varieties from the introduction, but in this concrete setting. 

Let $E$ be an elliptic curve defined over $K$ and let $\Lambda$ be the associated complex lattice (see \cite[Chapter 6, Theorem 5.1]{Silverman}). Let $\wp$ denote the Weierstrass $P$-function associated to the lattice $\Lambda$, and let $\wp'$ be its derivative. It is well-known that 
(see e.g., \cite[Chapter 6, Theorem 3.6]{Silverman})  that
the map $$z \mapsto [\wp(z):\wp'(z):1]$$ defines a complex analytic isomorphism of complex Lie groups between $\C/\Lambda$ and $E(\C) \subset \P^2(\C)$.

As $E$ is defined over $K$ and hence over $\R$, this map restricts to the usual exponential map $$\exp_E: \R \to E(\R) \subset P^2(\R), \quad z \mapsto [\wp(z):\wp'(z):1]$$ whose image is the connected component of the origin, $E(\R)^0$, and which intertwines the group law of addition on $\R$ with the group law on the set $E(\R)$ (since the curve is defined over $K$, the group law preserves real coefficients, hence restricts to a group law on $E(\R)$). The kernel of this map is given by $\Z \omega$ for some nonzero period $\omega$. This map distorts the Riemannian metric on $E$ only by a bounded constant.  

Let $P$ be a point in $E(\R)^0 \setminus E(K)$ and let $Q$ be a generator of the free summand of the Mordell--Weil group $E(K)$ (we need only one such point because we assume that $E$ has rank 1). As in the discussion of Waldschmidt's conjecture above, let $\theta$ denote a preimage of $Q$ under $\exp_E$, and $\gamma$ a preimage of $P$.

\ignore{Let $\exp$ denote the usual exponential map $S^1 \cong \R/\Z\omega \to E(\R)^0$. Let us define $$\theta := \varphi(Q),\quad \gamma := \varphi(P).$$ }
\begin{lemma}\label{lem1}
We have $\frac{\theta}{\omega} \notin \Q$ and there do not exist integers $a,b$ for which $a \frac{\theta}{\omega} + b = \frac{\gamma}{\omega}$. 
\end{lemma}
\begin{proof}
If it were that $\frac{\theta}{\omega} = \frac{p}{q} \in \Q$, then $q\theta = p\omega$, and hence $$[q]Q = [q]\exp(\theta) =  \exp(q\theta) = \exp(p\omega)= [p]\exp(\omega) = [p]O = O,$$ contradicting that $Q$ has infinite order. 

Similarly, if there existed integers $a,b$ for which $a\frac{\theta}{\omega} + b = \frac{\gamma}{\omega}$, then we would have that $a\theta + b\omega = \gamma$, and hence $$P = \exp(\gamma) = \exp(a\theta + b\omega) = [a]\exp(\theta) + [b]\exp(\omega) = [a]Q + [b]O = [a]Q,$$ contradicting that $P \in E(\R)\setminus E(K)$ (actually, it is enough that $P$ not belong to the orbit of $Q$; it could be a torsion rational point, for instance). So no such integers can exist. 
\end{proof}

Because the exponential map is an isometry, we have reduced the problem of approximating $P$ by the orbit $[q]Q$ to one of approximating $\gamma = \varphi(P)$ by points of the form $q\theta = \varphi([q]Q)$ in $\R/\Z\omega$. In other words we are interested in how small expressions of the form $$|\gamma - q\theta + p\omega|$$ can be in terms of the size of $|q|$. Dividing this expression through by $|\omega|$, we see that this is precisely a problem in \emph{inhomogeneous Diophantine approximation} on the real line. We have the following classical theorem of Minkowski: 
\begin{theorem*}[Minkowski; see for instance \cite{Niven}]
Let $\frac{\theta}{\omega}$ be any irrational number and let $\frac{\gamma}{\omega}$ be any number for which $\frac{\gamma}{\omega}= a\frac{\theta}{\omega} + b$ has no solutions in integers $a,b$. Then there exist infinitely many pairs of integers $(q,p)$ for which $$|q| \left|q \frac{\theta}{\omega} + p - \frac{\gamma}{\omega}\right|< \frac1{4}.$$ 
\end{theorem*}

Given the results of Lemma \ref{lem1}, we see that for infinitely many $n \in \Z$ we have $$d_E(P,[q]Q) \asymp |\gamma - q\theta + p\omega| < \frac{|\omega|}{4|q|}.$$ Since the canonical height satisfies  $\hat{h}([q]Q) = q^2 \hat{h}(Q)$ (see \cite[Chapter 8, Theorem 9.3(b)]{Silverman}), there therefore there exists a constant $C$, depending only on the curve $E$ for which $$d_E(P,[q]Q) < C(\hat{h}([q]Q))^{-1/2}.$$ As noted previously, this establishes the inequality $$\sigma_{E,K} \geq \frac1{2}.$$ 

Taken in conjunction with Lemma \ref{exp_UB} above, with $\frac{r}{2g} = \frac12$, this proves that for any rank 1 elliptic curve over $K$, $\sigma_{E,K} = \frac12$.

\ignore{
\subsection{Higher Rank Dirichlet, Optimality}
\subsubsection{Dirichlet-type Theorem}

One could hope to get a (weak) Dirichlet theorem for the general higher rank case in the same way we obtained Theorem \ref{weakDT} above. The analogous statement would require showing that, using the same notation as above, that for every \emph{totally irrational} tuple $(\matrixA_1,..,\matrixA_r)$ which are the generators of some elliptic curve over $\Q$, and every $\gamma$, that there are infinitely many $\textbf{n} \in \Z^r$ satisfying 
\[
\min_{m \in \Z} \left|\sum^r_{i=1} n_i\matrixA_i + m\omega -\gamma\right| < \frac{C}{\displaystyle\max_{1 \leq i \leq r} |n_i|^r}
\]
for some constant $C$, depending only on $\gamma$ and perhaps $(\matrixA_1,...,\matrixA_r)$ (in our Theorem \ref{weakDT} this constant is not dependent on $\boldsymbol{\matrixA}$, which we refer to as a \emph{uniform} Dirichlet result).  This would follow immediately if, for instance, this statement was simply true \emph{for all} totally irrational tuples. Unfortunately, that is not the case as the following theorem demonstrates: 
\begin{theorem*}[{\cite[Theorem 1]{BugeaudChevallier}}]\label{Dir_counterex}
Let $r=2,3$. There exist totally irrational $r$-tuples $(\gamma_1,...,\gamma_r)$ for which the set 
\[
\left\{\gamma\in \R\ :\ \min_{m \in \Z} \left|\sum^r_{i=1} n_i \gamma_i +m - \gamma\right| < \frac{C}{\displaystyle\max_{1 \leq i \leq r} |n_i|^r} \ \text{holds for some $C> 0$ and infinitely many } (n_1,...,n_r)\in \Z^r \right\}
\]
has Hausdorff dimension exactly $\frac1{r}$. \comkeith{Make sure that less than is correct here. }
\end{theorem*}
\emph{What does this mean for us?} If such a tuple could arise as the image of the generator of an elliptic curve defined over $\Q$, then for that curve, the analogous equation \eqref{dirichlet} would not hold. The conjecture that $\sigma_E \geq \frac{r}{2}$ might still be true, but we would need a different proof.

However, by a theorem of Cassels (\cite[Theorem XVIII]{Cassels}), the set of such \comkeith{Be more specific about the Cassels statement here} tuples is a null set with respect to the Lebesgue measure in any dimension. So the question is whether the generators of elliptic curves can hit this null set. Since the collection of such generators is a countable collection, measure theoretic results will not shed any light on these problems.  \comdavidl{The best I have found so far regarding the diophantine properties of the generators is a work by Gaudron on commutative algebraic groups (see https://theses.hal.science/tel-00001165v1 Theorem }

\subsubsection{Optimality}
The question of the optimality of such a theorem, even if it holds, is trickier. It follows from  \cite[Theorem 1]{BHKV} that the set 
\[
\left\{\gamma\in \R\ :\ \min_{m \in \Z} \left|\sum^r_{i=1} n_i \gamma_i +m - \gamma\right| > \frac{c}{\displaystyle\max_{1 \leq i \leq r} |n_i|^r} \ \text{holds for some $c:= c(\gamma) > 0$ and all } (n_1,...,n_r)\in \Z^r \right\}
\]
has full Hausdorff dimension. This was used in the proof of Theorem \ref{weakDT} to establish the bound $\sigma_E \leq \frac1{2}$. \textcolor{violet}{Similar results hold for other abelian varieties (except for the case of rank 1 elliptic curves).  In particular, for a higher rank or higher dimensional version of Theorem \ref{weakDT} to hold the resulting matrix would have to be not-very-well approximable.  This would require the generators to miss a set of dimension $mn\left(1-\frac{1}{m+n}\right)$ \cite[Theorem 1.11]{dasfishman}.}

In the higher rank setting, however, there is a complication with using the canonical height. The canonical height is, as mentioned earlier, a quadratic form on $E(\Q)$. This means that it satisfies the parallelogram law 
\[
\hat{h}(P+Q)= 2\hat{h}(P) + 2\hat{h}(Q) - \hat{h}(P-Q) .
\]
Therefore, by induction, we have that 
\[
\sqrt{\hat{h}\left(\sum^r_{i=1} [n_i] Q_i\right)} \leq \sqrt{2\sum^r_{i=1} n^2_i \hat{h}(Q_i)} \lesssim_\times  \displaystyle\max_{1\leq i \leq r} |n_i|.
\]
Obtaining a similar lower bound in terms of $\displaystyle\max_i |n_i|$ is harder though, since it requires an upper bound on the magnitude of terms like $\hat{h}(P-Q)$, which is a highly nontrivial problem. \comdavid{If we know that the canonical height is a quadratic form, doesn't that automatically imply the lower bound? Why invoke the parallelogram law at all?}

\subsection{Higher Rank Conjecture}

Even if the analogous Dirichlet-type theorem for higher rank is not true, we can still try to establish Conjecture \ref{DEconj} directly. This would follow immediately if for every $\varepsilon > 0$ there exist infinitely-many $\textbf{n} \in \Z^r$ satisfying
\[
\min_{m \in \Z} \left|\sum^r_{i=1} n_i\matrixA_i + m\omega -\gamma\right| < \frac{C}{\displaystyle\max_{1 \leq i \leq r} |n_i|^{r-\varepsilon}}.
\]

This is a problem in classical Diophantine approximation on the real line.

}

\ignore{
\begin{theorem}\label{DaniTransference}
Fix $m,n\in\N$ and let $\matrixA \in \MM_{m\times n}$. Then $\matrixA^T$ is not $\psi$-approximable for $0\leq\lim_{x\to\infty} x^\frac{n}{m}\psi(x)<\infty$ if and only if there is a $C>0$ such that for all sufficiently large $Q$ and all $\boldsymbol\gamma$, there exist $\pp\in\Z^m$ and $\qq\in\Z^n$ with $0 < \|\qq\| \leq Q$ satisfying
\[
\|\matrixA\qq + \pp + \boldsymbol\gamma\| < C\varphi\left(\frac{Q}{C}\right),
\]
where $\varphi(e^{t/n+r(t)})=e^{-t/m+r(t)}$.  \comdavidl{This is similar to the function defined in \cite{KleinbockWadleigh}[Lemmas 3.1,3.4].  Also check \url{https://arxiv.org/abs/2503.21180v1}.  Equivalently, $\psi^{-1}\left(\frac{C}{Q}\right)\|\matrixA\qq + \pp + \boldsymbol\gamma\| < C$.  In particular, the relation between $\psi$ and $\varphi$ is the same as between $f$ and $g$ in the Moshchevitin paper Eqns. 2, 3.  We probably only need that $\psi$ is decreasing and invertible as in their paper.  Corollary 1 of section 3.3 is an if and only if statement for BA (but with a different statement than ours).  I think that their Theorem A and Theorem 4 together basically show what we do here.}
\end{theorem}
\begin{proof}
To prove the forward direction, assume that $\matrixA^T\in\MM_{n\times m}$ is not $\psi$-approximable. Consider the unimodular lattice 
\[
\Lambda_\matrixA^* := \begin{pmatrix} I_m&0\\ -\matrixA^T&I_n\end{pmatrix} \Z^{m+n}
\]
and the flow on the space of lattices given by 
\[
g_t := \begin{pmatrix} e^{t/m} I_m &0\\0&e^{-t/n} I_n\end{pmatrix}.
\]
By the definition of not $\psi-$approximable, it follows that all of the nonzero lattice vectors in $g_{-t} \Lambda_\matrixA^*$ have supremum norm greater than $e^{-r(t)}$, where $\psi(e^{t/m-r(t)})=e^{-t/n-r(t)}$.  So, $\lambda_1^*\geq e^{-r(t)}$.

Let $\mu(\Lambda)$ be the codiameter in the infinity norm of the lattice $\Lambda$.  Let $\{\bb_i\}_{i=1}^d$ be a basis for $\Lambda$
such that $|\bb_i|_\infty=\lambda_i$.  For any point $\mathbf{x}\in\R^d,$ there is a
point $\qq\in\Lambda$ within $|\mathbf{x}-\qq|_\infty\leq\sum_{i=1}^d(|x_i-q_i\|\bb_i|_\infty)\leq\sum_{i=1}^d(
|x_i-q_i|\lambda_i)\leq\frac{1}{2}\sum_{i=1}^d\lambda_i\leq\frac{d}{2}\lambda_d$.
By a theorem of Mahler \cite[end of section 3]{Mahler_Transference}, $\lambda_d\asymp1/\lambda_1^*$.  So, along with the lemma above, this implies $\mu(\Lambda)\asymp1/\lambda_1^*\leq e^{r(t)}$.

Now consider the affine grid obtained by translating a lattice $\Lambda$ by some vector $\yy$. In the quotient space $\R^{m+n}/\Lambda$, every element of the grid $\Lambda + \yy$ is mapped to the projection of the vector $\yy$. Distances in the projection are given by 
\[
d_{\R^{m+n}/\Lambda}(\xx,\yy) = \inf_{\rr \in \Lambda}\|\xx - \yy + \rr\|.
\]
This means that the length of the shortest nonzero vector in the affine grid is equal to the distance from $\0$ to $\yy$ in the quotient space $\R^{m+n}/\Lambda$, which is bounded from above by the codiameter of $\Lambda$. 

Combining this observation with the argument above implies that for any $\boldsymbol\gamma\in\R^m$ the affine grid 
\[
g_t\Lambda_\matrixA + \begin{pmatrix} \boldsymbol\gamma\\\0 \end{pmatrix}
\] 
has shortest nonzero vector shorter than $C_2e^{r(t)}$, for some uniform constant $C_2$ (depending only on the dimension of the lattice). 

Fix $Q$ as in the statement of the theorem, and let $t$ satisfy $\frac{t}{n}+r(t) = \log(Q/C_2)$.
This is always possible as the left-hand side is nondecreasing \cite{KleinbockMargulis}[equation 8.2b].
Consider the grid $g_t \Lambda_\matrixA + \begin{pmatrix} e^{t/m} \boldsymbol\gamma \\ 0 \end{pmatrix}$. We know that there exists at least one nonzero element whose norm is less than $C_2e^{r(t)}$. All such elements of this grid can be written as $g_t u_\matrixA \begin{pmatrix} \pp \\ \qq \end{pmatrix} + \begin{pmatrix} e^{t/m} \boldsymbol\gamma \\ 0 \end{pmatrix}$, so we pick such a vector and compute its sup norm: 
\[
\begin{pmatrix} e^{t/m} I_m &0\\0& e^{-t/n} I_n\end{pmatrix} \begin{pmatrix} I_m&\matrixA\\ 0&I_n\end{pmatrix} \begin{pmatrix} \pp\\ \qq\end{pmatrix} + \begin{pmatrix} e^{t/m}\boldsymbol\gamma\\ 0 \end{pmatrix} = \begin{pmatrix} e^{t/m}\left(\matrixA\qq + \pp + \boldsymbol\gamma\right)\\ e^{-t/n} \qq \end{pmatrix},
\]
where both components of this vector are less than $C_2e^{r(t)}$ in absolute value. 

Looking at the second inequality first, namely $e^{-t/n}\|\qq\| < C_2e^{r(t)}$, we see that this means 
\[
e^{-t/n}\|\qq\| = \frac {C_2}{Q}e^{r(t)} \|\qq\| < C_2e^{r(t)} \Rightarrow \|\qq\| < Q. 
\]
By examining the first component, we similarly see that 
\[
C_2e^{r(t)} > \|e^{t/m}\left(\matrixA\qq + \pp + \boldsymbol\gamma\right)\| \Rightarrow C_2>\frac{1}{\varphi\left(\frac{Q}{C_2}\right)} \|\matrixA\qq + \pp + \boldsymbol\gamma\|
\]
which implies that $\|\matrixA\qq + \pp + \boldsymbol\gamma\| < C_2\varphi\left(\frac{Q}{C_2}\right)$.

To prove the complementary direction, we will actually prove the following: If $\matrixA^T$ is $\psi$-approximable, then the set \[ D:= \left\{\boldsymbol\gamma\in\R^m:\ \forall C>0,\ \exists Q_k \to \infty \;\forall k \; \forall \pp\in\Z^m,\qq\in\Z^n \, 0 < \|\qq\|< Q_k,\ \|\matrixA\qq + \pp + \boldsymbol\gamma\| \geq C\varphi\left(\frac{Q_k}{C}\right)  \right\} \] is hyperplane absolute winning for any $\beta<\frac13$. \comdavid{need to define hyperplane absolute winning} Since any hyperplane absolutely winning set is in fact uncountable and has full Hausdorff dimension, this is a priori stronger than simply proving that the set is nonempty. 

Since $\matrixA^T$ is $\psi$-approximable \comdavidl{may need to alter the definition of $\psi$-approximable for this}, for each $k$, there exist lattice points $(u_\matrixA^{-1})^T \begin{pmatrix}\pp_k\\\qq_k\end{pmatrix} \in \Lambda_\matrixA^*$ for which 
\[
\displaystyle\min_{t_k} \left\|e^{r(t_k)} g_{-t_k}(u_\matrixA^{-1})^T \begin{pmatrix}\pp_k\\\qq_k\end{pmatrix}\right\| \leq \frac1{k}.
\]
Let $t_k$ denote the unique time at which $\rr_k = \begin{pmatrix}\aa_k\\\bb_k\end{pmatrix} := g_{-t_k} (u_\matrixA^{-1})^T \begin{pmatrix}\pp_k\\\qq_k\end{pmatrix}$ satisfies $\|\aa_k\| = \|\bb_k\|$ (by Lemma \ref{SimulMin} these are the same $t_k$'s as in the above equation).
Without loss of generality, by extracting a subsequence if necessary, we can assume that 
\[
r(t_k)-r(t_{k+1})+\frac{t_{k+1}-t_k}{m} > -2\ln(\beta)
\]
for all $k$ \cite[equation 8.2a]{KleinbockMargulis}.

We now describe Alice's strategy as a response to Bob's choice of ball. She deletes a hyperplane based on the following criterion:
\begin{center}
If for some $k$, the radius of Bob's ball lies between $4C\beta^{-1} e^{r(t_k)-t_k/m}$ and $4C\beta^{-2}e^{r(t_k)-t_k/m}$,
\end{center}
then Alice will delete a specific hyperplane-neighborhood (to be determined shortly), and otherwise she can choose arbitrarily. So, it remains only to explain how Alice will make her choice when that criterion is satisfied and why such an outcome will belong to $D$.

Since $\rr_k\in(g_{t_k} \Lambda_\matrixA)^* = g_{-t_k}(u_\matrixA^{-1})^T \Z^{m+n}$, for all $\begin{pmatrix}\cc\\\dd\end{pmatrix}\in g_{t_k} \Lambda_\matrixA$ we have $\aa_k\cdot\cc + \bb_k\cdot\dd \in \Z$. Now for each $i\in\Z$, let $H_{k,i}' = \{\begin{pmatrix}\cc\\\dd\end{pmatrix}\in\R^{m+n} : \aa_k\cdot\cc + \bb_k\cdot\dd = i\}$ and $H_{k,i} = \{\cc\in \R^m : \begin{pmatrix}\cc\\\0\end{pmatrix} \in H_{k,i}'\}$. Note that the distance between $H_{k,i}$ and $H_{k,i+1}$ is $1/\|\aa_k\| > ke^{r(t_k)}$. So when $k$ is sufficiently large, Bob's ball only intersects at most \textit{one} of the hyperplane-neighborhoods $N(e^{-t_k/m}H_{k,i},2C e^{r(t_k)-t_k/m})$. Indeed, the distance between two such intervals is at least $(k - 4C)e^{r(t_k)-t_k/m} \asymp ke^{r(t_k)-t_k/m}$, whereas Bob's ball is only of width at most $\frac{8Ce^{r(t_k)-t_k/m}}{\beta^2} \asymp e^{r(t_k)-t_k/m}$. Therefore Bob's choice can intersect only at most one such hyperplane-neighborhood, and Alice can delete it when appropriate. 

So, it remains to show that the outcome of the game $\yy$ lies in $D$. \comkeith{If $y \notin D$ then a certain inequality has to hold for all $Q$ sufficiently large. Find $k$ such that $t_k > \log(Q_0)$, then proceed as follows. } Fix $C \geq 1$ and let $\frac{Q}{C} = e^{r(t_k)+\frac{t_k}{n}}$. If there exist $\pp,\qq$ with $\|\qq\|< Q$ for which $\|\matrixA \qq + \pp - \yy\| < C\varphi\left(\frac{Q}{C}\right)$, then we will show that $\yy$ was deleted by Alice during one of her moves. Consider the point $\yy_k := g_{t_k}\begin{pmatrix}\yy\\\0\end{pmatrix} = \begin{pmatrix}e^{t_k/m}\yy\\\0\end{pmatrix}$, and its distance to the lattice $\Lambda_k := g_{t_k} \Lambda_\matrixA$. 
\[
\dist(\yy_k, \Lambda_k) \leq \max\{e^{t_k/m}\|\matrixA\qq + \pp - \yy\|, e^{-t_k/n}\|\qq\|\} < \max\{Ce^{r(t_k)}, Ce^{r(t_k)}\} = Ce^{r(t_k)}.
\]
Therefore
\[
\dist(\yy_k,H_{k,i}') < Ce^{r(t_k)} \text{ for some $i\in\Z$.}
\]
Say for example $\dist(\yy_k,\begin{pmatrix}\cc\\\dd\end{pmatrix}) < Ce^{r(t_k)}$ for some $\begin{pmatrix}\cc\\\dd\end{pmatrix}\in H_{k,i}'$. Then $\aa_k\cdot\cc + \bb_k\cdot\dd = i$, so $\aa_k\cdot (\cc + \frac{\bb_k\cdot\dd}{\|\aa_k\|^2}\aa_k) = i$, i.e., $\cc + \frac{\bb_k\cdot\dd}{\|\aa_k\|^2}\aa_k \in H_{k,i}$. Since $\|\aa_k\| = \|\bb_k\|$, we have $\|\frac{\bb_k\cdot\dd}{\|\aa_k\|^2}\aa_k\| \leq \|\dd\| < Ce^{r(t_k)}$. So $\dist(\cc,H_{k,i}) < Ce^{r(t_k)}$ and thus
\[
\dist(e^{t_k/m}\yy,H_{k,i}) < 2Ce^{r(t_k)}.
\]
Therefore, by the above equation, the distance between $\yy$ and $e^{-t_k/m}H_{k,i}$ is less than $2Ce^{r(t_k)-t_k/m}$, in which case Alice would have deleted the hyperplane-neighborhood $N(e^{-t_k/m}H_{k,i},2Ce^{r(t_k)-t_k/m})$ containing $\yy$, a contradiction. 

To finish the proof, we recall that the hyperplane absolute game has the countable intersection property discussed above \comkeith{Discuss it above}, which shows that we can do this for every integer $C\geq 1$, and the intersection of all of these sets, which is precisely $D$, is still absolutely winning. 
\end{proof}
\begin{proof}
Alternative reverse:
Assume for every $Q$ sufficiently large, and for all $\boldsymbol\gamma\in\mathbb{R}^m$ there is a $\qq\in\Z^n$ and an $\pp\in\Z^m$ such that
\[
\|\qq\|<Q
\]
and
\[
\|A\qq - \pp -\boldsymbol\gamma\| < C\varphi\left(\frac{Q}{C}\right).
\]
Define
\[
\frac{t}{n}+r(t)=\ln\left(\frac{Q}{C}\right)
\]
or
\[
Q=Ce^{\frac{t}{n}+r(t)}.
\]
This implies that
\begin{align*}
\max\left(e^\frac{t}{m}\|A\qq - \pp -\boldsymbol\gamma\|,e^{-\frac{t}{n}}\|\qq\|\right)
&< \max\left(e^\frac{t}{m}C\varphi\left(\left(\frac{Q}{C}\right)^n\right)^{1/m},e^{-\frac{t}{n}}Q\right)\\
&=Ce^{r(t)}.
\end{align*}
So that the RHS of this equation is an upper bound on the lattice's codiameter $g_t \Lambda_A$, and thus by Lemma \ref{lemmacodiamest} we have
\[
\lambda_{m+n}< 2Ce^{r(t)}.
\]
By a theorem of Mahler \cite{Mahler_Transference}, $\frac{1}{\lambda_1^*}\asymp\lambda_{m+n}$
Therefore there is a $c_1>0$ (depending only on the dimension of the lattice) such that
\[
\lambda_1^*\geq \frac{c_1}{\lambda_{m+n}}> \frac{c_1}{2C}e^{-r(t)}.
\]
So that for all $t>t_0$,
\[
\Delta(g_{-t}\Lambda_A^*)< r(t)+\ln\left(\frac{2C}{c_1}\right).
\]
By the Dani correspondence theorem, this implies that $A^T$ is not $\psi-$approximable.
\end{proof}

\begin{corollary}
Let $A$ be an $m\times n$ badly-approximable matrix.  Then $A$ is badly approximable if and only if
\[
\|\matrixA\qq + \pp + \boldsymbol\gamma\| < \frac{C}{Q^\frac{n}{m}}.
\]
\end{corollary}
\begin{proof}
By Khintchine's transference theorem \cite{Cassels}[Chapter 5] if $A$ is badly approximable, so is $A^T$.  Apply \ref{DaniTransference} with $\psi(x)=\frac{1}{x^\frac{m}{n}}$.
A little algebra yields $r(t)=0$, and thus $\varphi(Q)=\frac{1}{Q^\frac{n}{m}}$.
\end{proof}
\begin{corollary}
Let $A$ be an $m\times n$ matrix.  Then $A$ is not-very-well-approximable if and only if for every $\varepsilon>0$
\[
\|\matrixA\qq + \pp + \boldsymbol\gamma\| < \frac{C}{Q^{\frac{n}{m}-\varepsilon}},
\]
\end{corollary}
\begin{proof}
By Khintchine's transference theorem if $A$ is not-very-well-approximable, so is $A^T$.  Fix $\varepsilon'>0$ and apply \ref{DaniTransference} with $\psi(x)=\frac{1}{x^{\frac{m}{n}+\varepsilon'}}$.
A little algebra yields $r(t)=\frac{\varepsilon' t}{m\left(1+\frac{m}{n}+\varepsilon'\right)}$, and thus $\varphi(Q)=\frac{1}{Q^{\frac{n}{m}-\varepsilon}}$, where $\varepsilon=\frac{n^2\varepsilon'}{m(m+n\varepsilon')},$ as required.
\end{proof}
}









\bibliography{bibliography}
\bibliographystyle{plain}

\end{document}